\documentclass[12pt,a4paper]{amsart}

\usepackage[
url=false,
giveninits=true,
isbn=false,
doi=false,
maxbibnames=99,
style=alphabetic,
sorting=nyt,
maxnames=4,
maxalphanames=4,
sortcites=none
]{biblatex}
\usepackage{amssymb}
\usepackage{enumerate}
\usepackage{mathtools}
\usepackage{xcolor}
\usepackage{textcomp}
\usepackage{amsthm}
\usepackage{thmtools} 

\usepackage[
colorlinks=true,
linkcolor=red!50!black,
citecolor=green!50!black,
urlcolor=blue!50!black,
hypertexnames=false
]{hyperref}
\usepackage[capitalise]{cleveref} 

\declaretheoremstyle[headformat=\NUMBER,headpunct=,]{claim}

\declaretheorem[name=Theorem,numberwithin=section,refname={Theorem,Theorems}]{Thm}

\declaretheorem[name=Lemma,sibling=Thm,refname={Lemma,Lemmas}]{Lem}
\declaretheorem[name=Corollary,sibling=Thm,refname={Corollary,Corollaries}]{Cor}
\declaretheorem[name=Conjecture,sibling=Thm,refname={Conjecture,Conjectures}]{conjecture}
\declaretheorem[name=Proposition,sibling=Thm,refname={Proposition,Propositions}]{Prop}

\declaretheorem[name=Notation,sibling=Thm,refname={Notation,Notation}]{Not}
\declaretheorem[name=Example,sibling=Thm,refname={Example,Examples}]{Ex}

\declaretheorem[style=claim,numberwithin=Thm,refname={,}]{claim}

\declaretheorem[name=Theorem]{introthm}

\DeclareMathOperator{\Aut}{Aut}
\DeclareMathOperator{\Fix}{Fix}
\DeclareMathOperator{\Frob}{Frob}

\DeclareMathOperator{\Inn}{Inn}

\DeclareMathOperator{\Out}{Out}

\DeclareMathOperator{\Stab}{Stab}

\DeclareMathOperator{\Syl}{Syl}

\DeclareMathOperator{\J}{J}

\DeclareMathOperator{\D}{D}
\DeclareMathOperator{\E}{E}
\DeclareMathOperator{\F}{F}
\DeclareMathOperator{\G}{G}
\DeclareMathOperator{\GL}{GL}
\DeclareMathOperator{\GU}{GU}

\DeclareMathOperator{\PGL}{PGL}
\DeclareMathOperator{\PGammaL}{P\Gamma L}

\DeclareMathOperator{\PSL}{PSL}
\DeclareMathOperator{\PSO}{PSO}
\DeclareMathOperator{\PSp}{PSp}
\DeclareMathOperator{\PSU}{PSU}
\DeclareMathOperator{\POmega}{P \Omega}
\DeclareMathOperator{\SL}{SL}

\DeclareMathOperator{\Sp}{Sp}
\DeclareMathOperator{\SU}{SU}
\DeclareMathOperator{\Sz}{{}^2B_2}

\DeclareMathOperator{\Alt}{Alt}

\DeclareMathOperator{\Sym}{Sym}

\DeclareMathOperator{\Gal}{Gal}
\DeclareMathOperator{\GF}{GF}

\DeclareMathOperator{\Ind}{Ind}
\DeclareMathOperator{\Irr}{Irr}

\DeclareMathOperator{\lcm}{lcm}

\renewcommand{\epsilon}{\varepsilon}
\renewcommand{\hat}{\widehat}

\newcommand{\ov}{\overline}

\DeclarePairedDelimiter{\abs}{\lvert}{\rvert}

\DeclarePairedDelimiter{\gen}{\langle}{\rangle}

\renewbibmacro{in:}{}
\DeclareFieldFormat[article]{title}{#1}

\title{Expansion of normal subsets of odd-order elements in finite groups }
\author{Chris Parker}
\author{Jack Saunders}
\address{Chris Parker\\
School of Mathematics\\
University of Birmingham\\
Edgbaston\\
Birmingham B15 2TT\\
United Kingdom}
\email{c.w.parker@bham.ac.uk}

\address{Jack Saunders\\
 School of Mathematics\\ University of Bristol\\ Fry Building\\ Woodland Road\\ Bristol\\ BS8 1UG,UK\\and\\ The Heilbronn Institute for Mathematical Research\\ Bristol\\ UK}
\email{J.P.Saunders@bristol.ac.uk}

\date{\today}
\begin{document}

\begin{abstract}
    Let $G$ be a finite group and $K$ a normal subset consisting of odd-order elements. The rational closure of $K$, denoted $\mathbf D_K$, is the set of elements $x \in G$ with the property that $\gen{x}= \gen{y}$ for some $y$ in $K$. If $K^2 \subseteq \mathbf D_K$, we prove that $\gen{K}$ is soluble.
\end{abstract}

\maketitle \pagestyle{myheadings}

\markright{{\sc }} \markleft{{\sc Chris Parker and Jack Saunders }}

\section{Introduction}
The study of products of conjugacy classes and normal subsets in finite groups has a long history with considerable recent activity \cite{Xu,AradHerzog,AradFisman,EllersGordeev, Gow, Garion, GuralnickNavarro, ClassPowers, Camina2020, Beltran2022, ParkerSaunders, ProductsNormalSubsets,LarsenTiep2}. Given a normal subset $K$ of a finite group $G$, the product
$$
    K^2 = \{mn \mid m, n \in K\}
$$
is also a normal subset of $G$. A natural question arises: how are the conjugacy classes that make up $K^2$ distributed in $G$? We refer to the way that $K^2$ gathers additional conjugacy classes as the \emph{expansion} of $K$.

At one extreme, Thompson's Conjecture, as formulated by Mazurov in 1984 (see \cite[\nopp9.24]{khukhro2025unsolvedproblemsgrouptheory}), asserts that in a finite simple group $G$, there exists a conjugacy class $C$ such that $G = C^2$. Thus in simple groups and, more generally, in non-soluble groups, we may expect extreme expansion of large conjugacy classes. Our main result, \cref{C3} below, supports this expectation by showing that if $C$ exhibits carefully restricted expansion, then $\gen{C}$ is soluble.

At the other extreme, Arad and Herzog conjectured in 1985 \cite[3]{AradHerzog} that if $G$ is a non-abelian finite simple group and $C, D$ are non-trivial conjugacy classes of $G$, then their product $CD$ is not a conjugacy class. In \cite[Theorem A]{GuralnickNavarro}, Guralnick and Navarro considered this question for an arbitrary finite group and showed that if $C$ is a conjugacy class in $G$ and $C^2$ is a conjugacy class, then $\gen{C}$ is a soluble normal subgroup of $G$. In \cite{ClassPowers}, Beltrán {\it et al.} consider conjugacy classes $C$ and $D$ and, for a fixed $n \ge 2$, they restrict the $n$th-power expansion of $C$ by insisting that $C^n = D \cup\{1\}$. They then prove that $\gen{C}$ is soluble. Furthermore, they conjecture that if $C^n=D\cup D^{-1}$, then $\gen{C}$ should be soluble. This conjecture remains open when $n=2$ \cite{Beltran2022}.

In this article, we investigate an arbitrary finite group $G$. Rather than working with conjugacy classes of $G$ we take $K$ to be a normal subset of $G$.
Let $x \in K$. Then
$$
    {\mathbf D}_x = \{y \in G \mid \gen{x} = \gen{y} \}.
$$
Define the \emph{rational closure} of $K$ to be
$$
    \mathbf D_K= \bigcup_{x \in K} {\mathbf D}_x.
$$
Notice that $\mathbf D_K$ is a normal subset of $G$ and that $\mathbf D_{\mathbf D_K} = \mathbf D_K$. Moreover, if $K$ is a union of rational conjugacy classes, then $K= \mathbf D_K$. We bound the expansion of $K$ by insisting $K^2\subseteq \mathbf D_K$.

We can now formulate our main theorem.

\begin{introthm}    \label{C3}
    Suppose that $G$ is a finite group and $K$ is a normal subset of $G$ consisting of elements of odd order. If $K^2 \subseteq \mathbf D_K$, then $\gen{K}$ is soluble.
\end{introthm}

Notice that the hypothesis of \cref{C3} allows the identity of $G$ to be a member of $K$ and does not require $K$ to be non-empty. Taking $K = G$ with $G$ a non-abelian simple group shows that the conclusion of \cref{C3} fails for normal subsets containing even-order elements. Focusing on just odd-order elements, if $K$ is the set of all elements of odd order in $G$ then $\mathbf D_K = K$. Thus, if $K^2 \subseteq \mathbf D_K$ then $K=\gen{K}$ is a subgroup of $G$ which has odd order. Therefore $\gen{K}$ is soluble by the Feit--Thompson Theorem. Hence \cref{C3} holds in this case. Similarly, if $K$ is a union of rational conjugacy classes then $K=\mathbf D_K$. Replacing $K$ by $K \cup \{1\}$ in this case we see that if $K^2\subseteq \mathbf D_K$ then $K$ is a subgroup. If, additionally, $K$ consists of odd-order elements, then we have $K=\gen{K}$ is soluble. This again demonstrates the truth of \cref{C3} for this special choice of $K$.

We present three corollaries to \cref{C3}. The first is a generalization of the Guralnick--Navarro Theorem to normal subsets of odd-order elements.

\begin{Cor} \label{Cor1}
    Suppose that $K$ is a normal subset of $G$ of odd-order elements and $L=\{x^2\mid x \in K\}$. If $K^2= L$ then $\gen{K}$ is soluble.
\end{Cor}

\begin{proof}
    For $x \in K$, as $x$ has odd order, $\gen{x^2}=\gen{x}$ and so $x^2 \in \mathbf D_K$. Hence $K^2=L\subseteq \mathbf D_K$ and the result follows from \cref{C3}.
\end{proof}

\begin{Cor} \label{Cor2}
    Suppose that $G$ is a group and $L$ is a normal subgroup of $G$ with $G/L$ abelian. Let $x \in G\setminus L$ have odd order and $K$ be the subset of odd-order elements in the coset $xL$. If $K^2$ contains only odd-order elements then $G$ is soluble.
\end{Cor}

The following corollary partially resolves the conjecture mentioned above and which appears in \cite[Conjecture 2]{ClassPowers}.

\begin{Cor} \label{CorCam}
    Suppose that $G$ is a finite group and $K$ and $D$ are conjugacy classes in $G$ with the elements in $K$ having odd order. If $K^n= D\cup D^{-1}$ for some $n \ge 2$, then $\gen{K}$ is soluble.
\end{Cor}

The paper is organized as follows. In \cref{sec:prel}, we present some preliminary results related to wreath products that will be used in the proof of \cref{C3}. The highlight is \cref{jack-3}, which demonstrates that the square of a certain conjugacy class expands to elements of smaller order in certain extensions of a $p$-group by a Frobenius group.

In \cref{sec:chars}, we follow \cref{sec:prel} with results which are derived using character-theoretic calculations. Perhaps the most interesting of these is \cref{coprime autos good}, which asserts that if, for example, $G=\SL_2(2^b)$ with $b$ a prime, and $x$ is an automorphism of $G$ of prime order $b$ not dividing $\abs{G}$, then regarding $G$ as $\Inn(G)$, we have
$$
    (x^G)^2 = x^2 G.
$$
That is, the square of the conjugacy class is the full coset of $G$. This can be regarded as a verification of a coset version of Thompson’s Conjecture for simple groups in the case $G = \SL_2(2^b)$ with $b$ prime (see \cref{conj:GT} below). The results of \cref{sec:chars} will be dominant when we come to consider rank 1 Lie-type groups in \cref{sec:rank1}.

In \cref{sec:simple}, we compile a substantial catalogue of facts about almost simple groups. For example, \cref{Cornice} asserts that in a simple classical group, if $P$ is the image of the parabolic subgroup that stabilizes a 1-space in the action on the natural module, then there exists an involution in every coset of $P$, except possibly $P$ itself. We wonder whether a similar result holds in the exceptional groups or for other  maximal parabolic subgroups. We surmise\footnote{We thank David Craven for pointing out the work of Cohen and Cooperstein \cite{CohenCoop}.} from \cite[W.6 Table 1]{CohenCoop}, that this is probably the case for the exceptional groups $\E_6(p^c)$ acting on its $27$-dimensional modules.

The proof of \cref{C3} begins in earnest in \cref{sec:red}. We assume that \cref{C3} is false and consider a minimal counterexample $(G, K)$. We  establish in \cref{lem:G/N cyclic} that the minimal counterexample $G$ has a unique minimal normal subgroup $N = N_1 \dots N_n$, where the $N_i$ are pairwise isomorphic non-abelian simple groups, and that for every $a \in K\setminus\{1\}$,
$$
    G = N \gen{a}.
$$
This then allows for a case-by-case investigation of the possibilities for $N_1$.

In \cref{sec:Alt}, we quickly eliminate the possibility that $N_1$ is an alternating or sporadic simple group. The latter case is handled mainly through a {\sc Magma} \cite{Magma} verification that in a sporadic simple group other than $\J_1$, for any odd-order element $g$, there exists an involution $t$ such that $gt$ has order 4 (\cref{spor-calcs}). Thus, for $a \in K$, we can find an involution $t$ such that $aa^t = atat$ is an involution, and hence is not in $\mathbf{D}_K$.

In \cref{sec:rank2}, we consider the case where $N_1$ is a Lie-type group of rank at least 2. Perhaps the most interesting aspect is that the hardest cases arise when $N_1 \cong \E_6(p^c)$ or ${}^2\!\E_6(p^c)$. In these cases, we rely on recent work on the  maximal subgroups of $\E_6(p^c) $ or ${}^2\!\E_6(p^c)$ by Craven \cite{Craven} and on the structure of the normalizers of maximal tori in these groups which is nicely presented in Javeed \emph{et al.} \cite{RowleyETAL}. Here is an instance where, if we could extend \cref{Cornice} to both $\E_6(p^c)$ and ${}^2\!\E_6(p^c)$, then the more extensive arguments required in this more complicated case could be removed.

\cref{sec:rank1} addresses the rank 1 Lie-type groups. As mentioned above, the character-theoretic calculations from \cref{sec:chars} play a crucial role here as in the most difficult cases $N_1$ is a minimal simple group and our inductive strategy briefly suggested above fails to deliver powerful conclusions.

Finally, \cref{sec:proof} synthesizes the various strands of our proof of \cref{C3} and establishes \cref{Cor2,CorCam}.

\subsection{Examples}

We now present some examples that demonstrate why we cannot include even-order elements in $K$ and also show that we cannot reasonably hope to say much more than $\gen{K}$ is soluble.

\begin{Ex}  \label{Ex:2}
    Suppose that $G= \PGammaL_2(8)$ and $L= F^*(G) \cong \PSL_2(8)$. Let $y\in G \setminus L$ and $K= yL$. Then $\mathbf D_K= G\setminus L$, $K^2=y^2L \subseteq \mathbf D_K$ and $G=\gen{K}$ is not soluble.
\end{Ex}

Fortunately, $K$ in \cref{Ex:2} has elements of order $6$. The next example suggests that Frobenius groups of odd order will generally provide examples of odd-order groups with relatively small normal subsets $K$ which are not conjugacy classes and satisfy $K^2\subseteq \mathbf D_K$.
\begin{Ex}
    Let $G= \gen{x,y}$ be a Frobenius group of order $21$ with $y$ of order $7$ and $x$ of order $3$. Then the set $K= x^G\cup y^G$ satisfies $K^2\subseteq \mathbf D_K= G \setminus \{1\}$.
\end{Ex}

The final example shows that determining the structure of $\gen{K}$ more precisely than saying it is soluble may be rather difficult.

\begin{Ex}
    In $G=(2\wr 3)\wr 5$, using {\sc Magma} \cite{Magma}, we have shown that there are conjugacy classes $I$ of elements of order $15$ and $J$ of elements of order $3$ such that, on setting $K= I \cup J$, we have $K^2 \subseteq \mathbf D_K$ and $\abs{G:\gen{K}}= 2$. In particular, $\gen{K}$ has even order.
\end{Ex}

We speculate that one can construct more complicated examples by taking iterated wreath products. It is tempting to guess that the square of a conjugacy class in a simple group must contain an element of even order. However, in $\Alt(5)$ both conjugacy classes of elements of order $5$ have square consisting of all the odd-order elements of $G$.

We close this subsection with some supporting evidence for the following conjecture.

\begin{conjecture}  \label{conj:GT}
    Suppose that $G =\Inn(G)$ is a non-abelian finite simple group and assume that $\alpha \in \Aut(G)$. Then there exists $x \in \alpha G$ such that $C=x^G$ satisfies $C^2=\alpha^2 G$.
\end{conjecture}

\begin{Lem} \label{Ex:extension}
    \cref{conj:GT} holds for $G$ a finite simple group found in the ATLAS \cite{ATLAS}, and for $G=\Alt(n)$, $n \ge 5$.
\end{Lem}

\begin{proof}
    For groups \(G\) in the {\sc ATLAS} \cite{ATLAS}, a {\sc Gap} \cite{GAP} calculation verifies the result.

    Suppose that $G= \Alt(n)$, $n \ge 7$ let $H=\Sym(n)$. Note that, if $n$ is even, $n$-cycles are odd and, if $n$ is odd, $(n-1)$-cycles are odd. Now application of \cite[Corollary 2.1]{Bertram} shows that if $n$ is even then the square of the class of $n$-cycles is $\Alt(n)$ and if $n$ is odd, then the square of the class of $(n-1)$-cycles is $\Alt(n)$.
\end{proof}

We also remark that work of Lev \cite{Lev} can be used to get information about products in $\PGL_n(q)$ when $q>4$ and $n\ge 2$.

\subsection{Notation}

Our notation follows that of \cite{Gorenstein} for group theory and \cite{Isaacs} for character theory. We hope that our naming of the simple groups is self-explanatory. For a group $Z$, $X \le Z$, and $y \in Y \subseteq X$, we define
$$
    Y^X = \{Y^x \mid x \in X\}
$$
as the set of $X$-conjugates of $Y$, and we denote the normal closure of $Y$ under the action of $X$ by $\gen{Y^X}$. We adopt the standard convention of writing $y^X$ and $\gen{y}$ instead of $\{y\}^X$ and $\gen{\{y\}}$, respectively. Additionally, we write $Y^\#$ for $Y \setminus \{1\}$. We follow \cite[Definition 2.5.13]{GLS3} for our definitions of the various types of automorphisms of groups of Lie type.

Throughout the paper, we shall claim that various facts may be verified using {\sc GAP} \cite{GAP} or {\sc Magma} \cite{Magma}. Where these claims are key to a proof then the relevant code (or a function to be applied to the groups in question) may be found at \cite{CodeGithub}.

\renewcommand{\abstractname}{Acknowledgements}
\begin{abstract}
    This work was supported by the Additional Funding Programme for Mathematical Sciences, delivered by EPSRC (EP/V521917/1) and the Heilbronn Institute for Mathematical Research.
\end{abstract}

\section{Preliminary lemmas}\label{sec:prel}

We begin with some group-theoretical lemmas which find their application in the proof of \cref{C3}. The goal of the first results is to locate products of conjugates that have a specified order, centralize an involution or complement a subgroup.

\begin{Lem} \label{lem:wreath orders}
    Let \(H\) be a group and \(C = \gen{x}\) cyclic of order \(n > 1\). Suppose that \(W \cong H \wr C\) with respect to the regular action of \(C\) and fix \(s\), \(t \in H\) with \([s,t]\) of order \(\ell\). Then, for \(w = (s, t^{-1}, 1, \ldots, 1) \in W\), we have that \((x^w x)^n\) has order \(\ell\).
\end{Lem}

\begin{proof}
    Let \(B = H_1 \times \dots \times H_n\) where \(H_i \cong H\) for all \(i\) and assume that for $(h_1, \ldots, h_n)\in B$ we have \((h_1, \ldots, h_n)^x = (h_n, h_1, \ldots, h_{n-1})\). Then,
    $$
        (h_1, \dots, h_n)^{x^2}= (h_{n-1},h_n,h_1, \dots, h_{n-2})
    $$
    and by gathering powers of $x$ on the left,
    $$
        ((h_1, \ldots, h_n)x^2)^n = ( h_1 h_3 \dots h_n h_2 \dots h_{n-1},\dots, h_nh_2\dots h_{n-2}).
    $$
    Now, with \(w = (s, t^{-1}, 1, \ldots, 1)\) as in the statement, we have that
    \[
        x^w x = w^{-1} w^{x^{-1}} x^2 = (s^{-1} t^{-1}, t, 1, \ldots, 1, s) x^2
    \]
    and so \((x^w x)^n = (s^{-1}t^{-1} s t, \dots, sts^{-1}t^{-1} )=([s, t], \dots,[s^{-1},t^{-1}])\) as required.
\end{proof}

\begin{Lem} \label{lem:complements}
    Suppose that $G$ is a finite group, $P=O_{2'}(G)$ and $\abs{G:P}=2$. Let $t\in G$ be an involution and assume that $[P,t]$ is abelian. If $x\in P\setminus[P,t]$, then $\gen{xx^t} \cap [P,t]= 1$ and $xx^t$ centralizes an involution in $G$.
\end{Lem}

\begin{proof}
    We have $xx^t= xtxt= (xt)^2$. Since $xt \not \in P$, $xt$ has even order. Let $s $ be the involution in $\gen{xt}$. Then $\gen{xt} \le C_G(s)$.

    Set $V=[P,t]$. Then $V$ is normal in $G=\gen{P,t}$. As $s=t^g$, for some $g \in G$ we have $V=V^g=[P,t]^g=[P,s]$. Also, since $V$ is abelian, $C_V(s)=1$. Hence $\gen{xx^t} \cap V\le C_G(s) \cap V=1$, as claimed.
\end{proof}

\begin{Lem} \label{lem:order help}
    Suppose that $G$ is a group, $A$ is a normal subgroup of $G$ and $x \in G^\#$. Set $\ell=\lcm\{\abs{ y^{\gen{x}}}\mid y \in A\}$. Then $xC_G(A)$ has order $\ell$ as an element of $G/C_G(A)$.
\end{Lem}

\begin{proof}
    Assume that $xC_G(A) $ has order $k$. For $y \in A$, $\abs{ y^{\gen{x}}}$ divides $k$, and so $\ell$ divides $k$. Assume that $\ell < k$ and that $r$ is a prime divisor of $k/\ell$. As $\gen{x}C_G(A)/C_G(A)$ has a unique subgroup $\gen{z}C_G(A)/C_G(A)$ of order $r$ and, for all $y \in A$, $ \abs{y^{\gen{x}}}$ is not divisible by the highest power of $r$ dividing $k$, we get $z C_G(A)=C_G(A)$. From this contradiction, we deduce \(\ell = k\).
\end{proof}

The next lemma is intended for use with \(z\) an involution so that, as seen in \cref{lem:complements}, \(azaz = a a^z\) is a product of two conjugates of \(a\).

\begin{Lem} \label{prod inv order}
    Suppose that $G$ is a finite group with a normal subgroup $N$ which is a direct product of $n\ge 2$ groups $N_1, \dots, N_n$. Let $a$ in $G$ be such that $\gen{a}$ permutes $\{N_1,\dots,N_n\}$ transitively by conjugation and set $b=a^n$. If $z\in N_1$ and $bz C_G(N_1)$ has order $m$ as an element of $N_G(N_1)/C_G(N_1)$, then $az$ has order $nm$.
\end{Lem}

\begin{proof}
    Let $t \in N_1$ be arbitrary. For $0 \le i \le n-1$, set $$t_{i+1}= t^{a^i}.$$ Then $t=t_1$ and $t_{i+1}\in N_{i+1}$ for $0\le i\le n-1$. As $z\in N_1$, $z$ centralizes $N_i$ for $1< i \le n$. Hence $$t^{(az)^k}= t^{a^k}= t_{k+1}$$ for $1 \le k \le n-1$ and so
    $$
        t^{(az)^n} = (t^{(az)^{n-1}})^{az} = (t_{n})^{az} = (t^{a^{n-1}})^{az} = t^{a^nz} = t^{bz}.
    $$
    As $bz$ has order $m$, $t^{(bz)^m} = t$ and so $t$ is centralized by $(az)^{nm}$. This holds true for all $t\in N_1$ and hence, setting $\ell = \operatorname{lcm}\{\abs{t^{\gen{bz}}} \mid t \in N_1\}$, $(az)^{n\ell}$ centralizes $N_1$. As $az$ acts transitively on $\{N_1, \dots, N_n\}$ and $C_G(N)=1$, we conclude that $(az)^{n\ell}=1$. However, $(az)^n C_G(N_1) = bzC_G(N_1)$ has order $\ell$ as an element of \(N_G(N_1)/C_G(N_1)\) by \cref{lem:order help}. Hence $\ell = m$ and $az$ has order $nm$.
\end{proof}

We reprise the proof of \cref{prod inv order} to explain the next lemma.

\begin{Lem} \label{prod inv normalizes}
    Suppose that $G$ is a finite group with a normal subgroup $N$ which is a direct product of $n\ge 2$ groups $N_1, \dots, N_n$. Let $a \in G$ be such that $\gen{a}$ permutes $\{N_1,\dots,N_n\}$ transitively by conjugation and set $b=a^n$. If $z\in N_1$ is such that $1\ne K \leq N_1$ is normalized by $bz$, then $\abs{K^{\gen{az}}}=n$.
\end{Lem}

\begin{proof}
    For $0 \le i \le n-1$, set $$K_{i+1}= K^{a^i}.$$ Then $K=K_1$ and $K_{i}\le N_{i+1}$ for $0\le i\le n-1$. As $z\in N_1$, $z$ centralizes $N_i$ for $1< i \le n$. Hence $$K^{(az)^k}= K^{a^k}= K_{k+1}$$ for $1 \le k \le n-1$ and so
    $$
        K^{(az)^n} = (K^{(az)^{n-1}})^{az} = (K_{n})^{az} = (K^{a^{n-1}})^{az} = K^{a^nz} = K^{bz} = K,
    $$
    as required.
\end{proof}

\begin{Thm} \label{our thm} 	
    Suppose $p$ is a prime, $G$ is a finite group, $A$ is a normal subset of elements of order $p$ in $G$ and every member of $A^2$ is a \(p\)-element. Then, setting $Q=\gen{A}$, $Q$ is soluble and, if $O_p(G)=1$, then $p$ is odd, $F(Q)$ is a non-trivial $p'$-group and $Q/F(Q)$ is an elementary abelian $p$-group.
\end{Thm}

\begin{proof}
    This is \cite[Theorem A]{ParkerSaunders}.
\end{proof}

\begin{Lem} \label{jack-1}
    Suppose that $S$ is a $p$-group and $Q$ an abelian normal subgroup of $S$. If $w \in S$ and $S= Q\gen{w}$, then $w^S=wS'$.
\end{Lem}

\begin{proof}
    Since $S/S'$ is abelian, $w^S \subseteq wS'$. We have
    $$
        S'=[ Q\gen{w}, Q\gen{w}]=[ Q , \gen{w}].
    $$
    As the map $Q\rightarrow [Q,\gen{w}]$ given by $q\mapsto [q,w]$ is a epimorphism, we have $\abs{Q:C_Q(w)}=\abs{[Q,\gen{w}]}$. Hence $\abs{w^S}=\abs{S:C_S(w)}=\abs{[ Q , \gen{w}]}=\abs{S'}$ and so $w^S=wS'$.
\end{proof}

\begin{Lem} \label{jack-2}
    Suppose that $S$ is a $p$-group and $Q$ an elementary abelian normal subgroup of $S$ with $S=Q\gen{w}$. Then $$\abs{\{uS'\mid uQ=wQ\}}= \abs{Q/S'}.$$
\end{Lem}

\begin{proof}
    The map
    \begin{eqnarray*}
        Q/S'    &   \rightarrow &   S/S'    \\
        qS'     &   \mapsto     &   wqS'
    \end{eqnarray*}
    is injective with image $\{uS'\mid uQ=wQ\}$.
\end{proof}

\begin{Prop}    \label{jack-3}
    Suppose that $n \ge 1$, $p$ is a prime, $G$ is a group, $S \in \Syl_p(G)$, $Q=O_p(G)$ is elementary abelian and $S/Q$ is cyclic of order $p^n$ with abelian normal $p$-complement $K/Q$ in $G/Q$. Assume that
    \begin{eqnarray}   \label{eqn}
        \text{for all } X \text{ with } 1&\ne X \le C_Q(S), \text{ we have }C_G(X)=S.
    \end{eqnarray}
    If $w\in S$ has order $p^{n+1}$, then $(w^G)^2$ contains elements which do not have order $p^{n+1}$.
\end{Prop}

\begin{proof}
    Because $K/Q$ is a normal $p$-complement in $G/Q$, $G=KS$ and as $n\ge 1$ and $Q=O_p(G)$, $S$ is not normal in $G$. As $C_Q(K)$ is normalized by $S$, if $C_Q(K) \ne 1$ then $C_{C_Q(K)}(S) \ne 1$. But then $K \le C_G(C_{C_Q(K)}(S)) = S$ by property (\ref{eqn}), contrary to $G=KS\ne S$. We record

    \begin{claim}   \label{jclm1}
        $C_Q(K)=1$.
    \end{claim}
    \medskip

    As $K/Q$ is a normal $p$-complement in $G/Q$, $K/Q$ is a $p'$-group and, as $K$ is abelian, $G$ is soluble. Let $L$ be a Hall $p'$-subgroup of $G$. Then $L$ is abelian, $K=LQ$ and
    $$
        G= N_G(L)K=N_G(L)Q
    $$
    by the Frattini argument. As $[N_Q(L),L]\le Q\cap L=1$ and $C_Q(L)=C_Q(K)=1$ by \cref{jclm1}, we have $N_G(L) \cap Q=N_Q(L) = 1$. That is, $N_G(L)$ is a complement to $Q$ in $G$. In particular, $N_G(L)$ has cyclic Sylow $p$-subgroup $\gen{x}$ of order $p^n$ and we may assume that $S= Q\gen{x}$.

    Let $G$ be a minimal counterexample to the proposition. Then $(w^G)^2$ consists of elements if order $p^{n+1}$. In particular, $w^2$ has order $p^{n+1}$ and so $p$ is odd and $w$ has order $p^{n+1}$. As $Q$ is elementary abelian and $w$ has order $p^{n+1}$ we have $S=\gen{w}Q$. Let $y \in \gen{x}$ have order $p$. Then $y \in N_G(L)$. We now claim the following.

    \begin{claim}   \label{jclm1.5}
        The subgroup $S$ is maximal in $G$. In particular, $S=N_G(S)$.
    \end{claim}

    \medskip

    Assume that $S< H < G$. Then
    $$
        H= H \cap G=H \cap LS= (H\cap L)S
    $$
    by the Dedekind Modular Law. Observe that $L > H \cap L \ne 1$. Set $R=O_p(H)$ and assume that $R> Q$. Then $y \in yQ \in R/Q$, as $S/Q$ is cyclic, and so $y \in R$. Hence, as $y \in \gen{x} \le N_G(L)$,
    $$
        [H\cap L,y] \le H\cap L \cap R=1.
    $$
    Thus $H\cap L$ is non-trivial and centralized by $y$.

     Set $H_1= [L,y]Q\gen{x}$. Then, as $1\ne H \cap L <L$ and $L$ is abelian, $[L,y]< L$ and $S<H_1 <G$. Because $L$ is abelian, $C_{[L,y]}(y)=1$ and therefore we have $O_p(H_1)= Q$ as otherwise the argument in the paragraph above applies. Hence $H_1$ satisfies the hypothesis of the proposition. Since $G$ is a minimal counterexample and $w\in S\le H$, $(w^{H_1})^2 \subseteq (w^G)^2$ contains elements which do not have order $p^{n+1}$, a contradiction. We conclude that $R=Q$ and consequently $H$ satisfies the hypothesis of the proposition, and we obtain $(w^{H})^2 \subseteq (w^G)^2$ contains elements of order other than $p^{n+1}$. Hence no such $H$ exists and consequently $S$ is maximal in $G$. As $S$ is not normal in $G$, we have that $S=N_G(S)$ is maximal in $G$, as required.

    \begin{claim}   \label{jclm2}
        Suppose that $S, T \in \Syl_p(G)$ are not equal. Then
        \begin{enumerate}[(i)]
            \item $G = \gen{S,T}$;
            \item $Q = S'T'$; and
            \item $\abs{S':S'\cap T'} = \abs{Q/S'}$.
        \end{enumerate}
    \end{claim}
\medskip

    Since $S$ is maximal in $G$ by \cref {jclm1.5}, (i) is clearly true.

    To prove (ii), notice first that $S'$ and $T'$ are both contained in $Q$. Assume that $Q>S'T'$. Then $Q/S'T'$ is centralized by $\gen{S,T}$. Since $L\le \gen{S,T}$ by (i), we have $C_Q(K)=C_Q(L) \ne 1$ by coprime action. This contradicts \cref{jclm1}. Hence (ii) holds and this immediately yields (iii) as $\abs{Q:S'}= \abs{Q:T'}$.

    \medskip

    Fix $g \in G$ so that $w^g \not \in S$. Define the map
    \begin{eqnarray*}
        \Psi: wS'&\rightarrow& G \\ws&\mapsto& w^gws
    \end{eqnarray*}
    where $s \in S'$.

    As $S'\le Q$ we know $wQ=wsQ$ and so $w^g wQ=w^g wsQ$. Hence the image of $\Psi$ is in the coset $w^g wQ$. Since $w^g w$ has order $p^{n+1}$, $w^g w Q\subseteq T$ for some $T \in\Syl_p(G)$. In particular, $T= Q\gen{w^g w}$ and the image of $\Psi$ is in $T$. By \cref{jack-1}, for $s_1,s_2\in S'$, $w^g ws_1$ and $w^gws_2$ are conjugate in $T$ if and only if $w^g ws_1T'= w^gws_2T'$. Since $T'$ is normalized by $w^gw$ and $S'$, $w^g ws_1 T'= w^g ws_2 T'$ precisely when $s_1 T'=s_2 T'$. Therefore $w^g ws_1$ and $w^g ws_2$ are conjugate in $T$ if and only if $s_1(S'\cap T')=s_2(S'\cap T')$. Hence the number of distinct $T$-conjugacy classes witnessed in the image of $\Psi$ is $\abs{S':S'\cap T'}= \abs{Q/S'}$ by \cref{jclm2}. By \cref{jack-2}, this is all of the classes with representative $b$ such that $T=\gen{b}Q$ and $bQ=w^gwQ$. Since there exists $t \in T$ of order $p^n$ such that $tQ=w^gwQ$, $tT'$ is one such class. But all the elements in $tT'$ are $T$-conjugate by \cref{jack-1} and so $tT'$ only has elements of order $p^n$. Therefore there exists $s \in S'$ such that $w^gws$ has order $p^n$. Since $ws$ is $S$-conjugate to $w$ by \cref{jack-1}, we conclude that $t\in (w^G)^2$. Hence $(w^G)^2$ has elements of order $p^n$, a contradiction.
\end{proof}

\section{Some character-theoretic considerations}\label{sec:chars}

This section is devoted to character-theoretic results which are essential for the proof of \cref{C3}.

\begin{Lem} \label{I3.12}
    Let $\chi\in \Irr(G)$ and $a,b \in G$. Then $\chi(a)\chi(b)=\frac{\chi(1)}{\abs{G}}\sum _{c\in G} \chi(ab^c)$.
\end{Lem}

\begin{proof}
    This is \cite[Exercise 3.12]{Isaacs}. Suppose that $\theta$ affords $\chi$. We calculate $\operatorname{Tr}(\theta(a\sum_{c\in G}b^c))$ in two ways. First we use the fact that $\theta(\sum_{c\in G}b^c)=\lambda I$ where $\lambda= \frac{\abs{G}\chi(b)}{\chi(1)}$ to obtain
    $$
        \operatorname{Tr}(\theta(a\sum_{c\in G}b^c))=\operatorname{Tr}(\theta(a)\theta(\sum_{c\in G}b^c)))= \operatorname{Tr}(\theta(a)\lambda I))= \frac{ \chi(a)|G|\chi(b) }{\chi(1)}.
    $$
    On the other hand, we have
    $$
        \operatorname{Tr}(\theta(a\sum_{c\in G}b^c))=\operatorname{Tr}( \sum_{c\in G}\theta(ab^c))=\sum_{c\in G}\chi(ab^c).
    $$
    This proves the claim.
\end{proof}

\begin{Lem} \label{HJ}
    Suppose that $a,b,c \in G^\#$. Then $a^Gb^G=c^G$ if and only if $\chi(a)\chi(b)=\chi(c)\chi(1)$ for all $\chi\in \Irr(G)$.
\end{Lem}

\begin{proof}
    This is \cite[Lemma 1]{MooriHung}, but we also record that it follows immediately from \cref{I3.12}.
\end{proof}

\begin{Lem} \label{char res}
    Suppose that $p$ is a prime, $G$ is a group, $T=O_p(G) \in \Syl_p(G)$ is elementary abelian and $T \le H \le G$ with $H\cap H^g=T$ for all $g \in G\setminus H$. Assume $R=C_T(H)$ has order $p$ and let $\chi $ be the inflation to $H$ of a non-trivial character of $R$ with kernel containing all the $p'$-elements of $H$. If $H=N_G(R)$, then $\psi=\Ind_{H}^G \chi $ is irreducible and $\psi(h)=1$ for all $h \in H$ of \(p'\) order.
\end{Lem}

\begin{proof}
    We have $\psi(1)=\abs{G:H}= \abs{G:C_G(R)}$. Suppose that $g \in G \setminus H$ and define $\rho= \chi^g_T$ to be the restriction of $\chi^g $ to $T$. Assume, aiming for a contradiction, that $\rho = \chi_T$. Since $T= R\times [T,H]$, we know $\ker\chi_T = [T,H]$, and $\ker \rho = [T,H^g]$. Hence
    $$
        [T,H] = \ker \chi_T = \ker\rho = [T,H^g]
    $$
    and so $[T,H]= [T,\gen{H,H^g}]$. Because $\abs{G/T}$ is coprime to $p$ and $T$ is abelian, we have $$T= C_T(\gen{H,H^g}) \times [T,\gen{H,H^g}]= C_T(\gen{H,H^g}) \times [T,H]=R\times [T,H]$$ and so we deduce that $R = C_T(\gen{H,H^g})$. Since $N_G(R)= H$ it follows that $H= H^g$. Hence $g \in N_G(H)=H$. We have proved that $\chi^g$ and $\chi$ are not equal for all $g \in G\setminus H$. Hence, as $\chi(1)=1$ and  $\{\chi^g_T\mid g \in G\}$ has size $\abs{G:H}$, Mackey's irreducibility criterion \cite[Section 7.4, Proposition 23]{SerreLinear} applies to yield $\psi$ is irreducible.

    Assume that $h \in H$ has $p'$-order, then $\chi(h)= 1$. Now, by definition,
    $$
        \psi(h)= \frac{1}{\abs{H}} \sum _{g\in G, h^g\in H}\chi(h^g)=\frac{1}{\abs{H}} \sum _{g\in H}\chi(h^g)=\chi(h) =1,
    $$
    as $H^g\cap H= T\in \Syl_p(G)$ for all $g \in G \setminus H$.
\end{proof}

\begin{Cor} \label{Even Order in K^2}
    Suppose that $G$ is a group, $T=O_2(G) \in \Syl_2(G)$ is elementary abelian, $T \le H \le G$, and $G/T$ is a Frobenius group with complement $H/T$. Assume $C_T(H)$ has order $2$, $N_G(C_T(H))=H$, $x \in H^\#$ has odd order and $K= x^G$. Then $K^2$ has elements of even order.
\end{Cor}

\begin{proof}
    Suppose that $K^2$ has no elements of even order. As $G/T$ is a Frobenius group, in $G/T$ we have $K^2T/T = (x^2)^GT/T$ and so, if every element of $K^2$ has odd order, we must have $K^2= (x^2)^G$. By \cref{char res}, there exists an irreducible character $\psi$ of $G$ with $\psi(x) = \psi(x^2) = 1$ and $\psi(1) = \abs{G:H} > 1$. However, because $K^2= (x^2)^G$, \cref{HJ} implies $\psi(x)^2= \psi(x^2)\psi(1)$. Since both statements cannot be simultaneously true, we deduce that $K^2$ must have elements of even order.
\end{proof}

\begin{Prop}    \label{coprime autos good}
    Assume that $G$ is isomorphic to one of $\PSL_2(2^b)$, $\PSL_2(3^b)$ or $\Sz(2^b)$ and $x \in \Aut(G)$ has order $b$ coprime to $\abs{G}$. If $L=G\gen{x}$, then \((x^L)^2 = x^2 G\).
\end{Prop}

\begin{proof}
    Let \(S = \gen{x}\). Because of \cite[Lemma 2]{GlaubermanCorreps}, the conjugacy classes of elements in \(x^2 G\) are in one-to-one correspondence with the conjugacy classes in \(C_G(S)\). For \(y \in C_G(S)\), we need to show that $x^2y \in (x^G)^2$. Hence by \cite[Theorem 4.2.12]{Gorenstein} it is sufficient to show that
    \[
        \sum_{\chi \in \Irr(L)} \frac{\chi(x)^2 \chi(x^{-2} y^{-1})}{\chi(1)} \neq 0.
    \]
    Let \(\Irr_S( G)\) denote the set of \(S\)-invariant characters of \(G\). Assume that \(\chi \in \Irr_S (G)\) and let \(\hat \chi\) be the canonical extension of \(\chi\) to \(L\) \cite[Lemma 13.3]{Isaacs}. By \cite[Theorem 13.6]{Isaacs} there exists a unique \(\beta_{\chi} \in \Irr ( C_G(S))\) such that, for $c \in C_G(S)$ and $w \in \gen{x}$ of order $b$, \(\hat \chi (cw) = \epsilon_{\chi} \beta_{\chi} (c)\) where \(\epsilon_{\chi} \in \{-1, 1\}\) is independent of $c$ and $w$. We also record that if $\chi$ is the trivial character, then $\epsilon _\chi=1$.

    Our plan is to use the fact that the character values of elements outside of $G$ are very small in comparison to the degrees of the characters of $G$.

    By \cite[Theorem 13.26 \& Corollary 6.17]{Isaacs}, the characters of \(L\) of degree coprime to \(b\) are \(\{\hat\chi \theta \mid \chi \in \Irr_S (G), \ \theta \in \Irr (S)\}\). Using the fact that the characters in \(\Irr ( S)\) are homomorphisms, we have
    \begin{align*}
        \sum_{\chi \in \Irr (L)} \frac{\chi(x)^2 \chi(x^{-2}y^{-1})}{\chi(1)}
            &= \sum_{\substack{\chi \in \Irr_S (G) \\ \theta \in \Irr (S)}} \frac{(\hat\chi \theta)(x)^2 (\hat\chi \theta)(x^{-2}y^{-1})}{\hat\chi(1) \theta(1)} \\
            &= \sum_{\substack{\chi \in \Irr_S (G) \\ \theta \in \Irr (S)}} \frac{(\hat\chi(x) \theta(x))^2 \hat\chi(x^{-2}y^{-1}) \theta(x^{-2})}{\chi(1)} \\
            &= \sum_{\substack{\chi \in \Irr_S (G) \\ \theta \in \Irr (S)}} \frac{(\epsilon_\chi)^3\beta_{\chi}(1)^2 \beta_{\chi}(y^{-1}) \theta(x)^2 \theta(x^{-2})}{\chi(1)} \\
            &= b\sum_{\chi \in \Irr_S (G)} \frac{\epsilon_\chi\beta_{\chi}(1)^2 \beta_{\chi}(y^{-1})}{\chi(1)}.\\
    \end{align*}
    In the case \(C_G(S) \cong \PSL_2(2)\), $y\in C_G(S)$ has order one of $1$, $2$ or $3$. Since $C_G(S)$ has $3$ irreducible characters, we have $\Irr_S(G)=\{1,\chi_1,\chi_2\}$ for some irreducible characters \(\chi_1\), \(\chi_2\). Set $a_i=\chi_i(1)$, with indices chosen such that \(\chi_0\) corresponds to the trivial character of \(C_G(S)\), \(\chi_1\) to the sign character and \(\chi_2\) to the unique faithful irreducible character. We also set $\epsilon_i= \epsilon_{\chi_i}$ for $i=1,2$. Thus we have
    \[
        \frac{1}{b} \sum_{\chi \in \Irr (L)} \frac{\chi(x)^2 \chi(x^{-2}y^{-1})}{\chi(1)} =
        \begin{cases}
            1 + \epsilon_{1}\frac{1}{a_1} +  \epsilon_{2}\frac{8}{a_2}      & \abs{y} = 1,  \\
            1 - \epsilon_{1}\frac{1}{a_1}                                   & \abs{y} = 2,  \\
            1 + \epsilon_{1}\frac{1}{a_1} -  \epsilon_{2}\frac{4}{a_2}      & \abs{y} = 3.
        \end{cases}
    \]
    As $b$ is coprime to $\abs{G}$, we have $b\ge 5$ and so $2^b\ge 32$. Thus $a_1$ and $a_2$ are at least $q-1\ge 31$ by \cite[Part II]{JordanSL2}.

    Similarly, when \(C_G(S) \cong \PSL_2(3)\), adapting the notation from the previous case and selecting a third root of unity $\omega$ we obtain the equations
    \[
        \frac{1}{b} \sum_{\chi \in \Irr (L)} \frac{\chi(x)^2 \chi(x^{-2}y^{-1})}{\chi(1)} =
        \begin{cases}
            1 + \epsilon_{1} \frac{1}{a_1} + \epsilon_{2} \frac{1}{a_2} + \epsilon_{3} \frac{27}{a_3}   & \abs{y} = 1,  \\
            1 + \epsilon_{1}\frac{1}{a_1} + \epsilon_{2}\frac{1}{a_2} - \epsilon_{3}\frac{9}{a_3}       & \abs{y} = 2,  \\
            1 +\epsilon_{1} \frac{\omega}{a_1} + \epsilon_{2} \frac{\omega^2}{a_2}                      & \abs{y} = 3,  \\
            1 + \epsilon_{1}\frac{\omega^2}{a_1} +\epsilon_{2} \frac{\omega}{a_2}                       & \abs{y} = 3.
        \end{cases}
    \]
    Since \(3^b \geq 243\), using \cite[Part III]{JordanSL2} we have that \(a_i \geq (3^b - 1)/2 \geq 121\) so that each expression above is non-zero.

    Finally, when \(C_G(S) \cong \Sz(2) \cong 5{:}4\), again adapting the notation from before we obtain
    \[
        \frac{1}{b} \sum_{\chi \in \Irr (L)} \frac{\chi(x)^2 \chi(x^{-2}y^{-1})}{\chi(1)} =
        \begin{cases}
            1 + \epsilon_{1} \frac{1}{a_1} + \epsilon_{2} \frac{1}{a_2} + \epsilon_{3} \frac{1}{a_3} + \epsilon_{4} \frac{64}{a_4}  & \abs{y} = 1,  \\
            1 + \epsilon_{1} \frac{1}{a_1} - \epsilon_{2} \frac{1}{a_2} - \epsilon_{3} \frac{1}{a_3}                                & \abs{y} = 2,  \\
            1 + \epsilon_{1} \frac{1}{a_1} - \epsilon_{2} \frac{i}{a_2} + \epsilon_{3} \frac{i}{a_3}                                & \abs{y} = 4,  \\
            1 + \epsilon_{1} \frac{1}{a_1} + \epsilon_{2} \frac{i}{a_2} - \epsilon_{3} \frac{i}{a_3}                                & \abs{y} = 4,  \\
            1 + \epsilon_{1} \frac{1}{a_1} - \epsilon_{2} \frac{1}{a_2} + \epsilon_{3} \frac{i}{a_3} - \epsilon_{4} \frac{16i}{a_4} & \abs{y} = 5.  \\
        \end{cases}
    \]
    If \(b = 3\) one may verify the claim with {\sc Magma} \cite{Magma} so we assume that \(b \geq 7\). Then, by \cite[Theorem 13]{Suzuki1} we have that \(a_i \geq q-1\) and so once again all of the above expressions are non-zero.
\end{proof}

For a group \(H\), let \(\mathbb{Q}_H\) denote a minimal splitting field for \(H\) over \(\mathbb{Q}\).

\begin{Lem} \label{CoprimeAutChars}
    Let \(X\) be a group and let \(\alpha \in \Aut X\) have prime order \(r\) not dividing \(\abs{X}\). Let \(G = X \rtimes \gen{\alpha}\) be the semidirect product of $X$ and $\gen{\alpha}$, $C= C_X(\alpha)$ and \(\Gamma = \Gal(\mathbb{Q}_G \colon \mathbb{Q}_X)\). Then for each \(1 \leq n < r\) there exists \(\rho \in \Gamma\) such that for all \(\chi \in \Irr(G)\) and all $c \in C$ we have that \(\chi^{\rho}(c\alpha) = \chi(c\alpha^n)\).
\end{Lem}

\begin{proof}
    Suppose that $\chi \in \Irr ( G)$. If $\chi$ restricted to $X$ is reducible, then $\chi$ is induced from an irreducible character of $X$ and so $\chi(y)=0$ for $y \in G\setminus X$. Otherwise, $\chi$ restricted to $X$ is $\gen{\alpha}$-invariant. For these characters, \cite[Theorem 13.6]{Isaacs} implies that $\chi(\alpha) \ne 0$. In particular, if $\chi(\alpha)=0$ for $\chi \in \Irr (G)$, then $\chi(y)=0$ for all $y\in G\setminus X$ and so the claim holds for such characters.

    Fix \(1 \leq n < r\) and \(\chi \in \Irr (G)\) which does not vanish on \(\gen{\alpha}\). Either \(\chi^{\rho} = \chi\) for all \(\rho \in \Gamma\) or \(\chi\) lies in a \(\Gamma\)-orbit of size \(r-1\). In the first case, \(\chi\) is the canonical extension of an \(\gen{\alpha}\)-invariant character of \(X\) and \(\chi\) takes the same value on all $c^C \alpha^k$ is independent of $k$ with $\alpha^k \ne 1$ by \cite[Theorem 13.6]{Isaacs} and the result holds for any choice of \(\rho \in \Gamma\). In particular, $\chi^\rho(c\alpha) = \chi(c\alpha^n)$ for these characters.

    Otherwise, using \cite[Theorem 13.26]{Isaacs} to determine the number of such characters and Gallagher's theorem \cite[Corollary 6.17]{Isaacs} to describe them, we have \(\chi = \psi \theta\) for \(\psi\) the canonical extension of some \(\gen{\alpha}\)-invariant character of \(X\) and \(\theta\) a linear character lifted from \(G/X\). Then \[\chi(c\alpha^n) = \psi(c\alpha) \theta(c\alpha^n) = \psi(c\alpha) \theta(\alpha)^n = \psi(c\alpha) \omega^n\] for some \(r\)th root of unity \(\omega=\theta(\alpha)=\theta(c\alpha)\). Since \(\Gamma\) acts transitively on the \(r\)th roots of unity in \(\mathbb{Q}_G\) and \(\psi(\alpha) \in \mathbb{Q}_X\) is fixed by \(\Gamma\), we may choose \(\rho\in \Gamma\) which maps \(\omega\) to \(\omega^n\) to obtain
    \[
        \chi^{\rho}(c\alpha) = \psi^{\rho}(c\alpha) \theta^{\rho}(c\alpha) = \psi(c\alpha) \omega^{\rho} = \psi(c\alpha) \omega^n = \chi(c\alpha^n)
    \]
    as required.
\end{proof}

\begin{Cor} \label{cor:one product enough}
    Let \(G\), \(X\), $C$, \(\alpha\) be as in \cref{CoprimeAutChars}. Assume that $t\in G$ and $c, d \in C$. If $$(d\alpha)^G \subseteq (t^G)(c\alpha)^G,$$ then for $1\le n< r$,
    $$
        (d\alpha^n)^G \subseteq (t^G)(c\alpha^n)^G.
    $$
\end{Cor}

\begin{proof}
    Using \cite[Theorem 4.2.12]{Gorenstein} we have
    $$
        \sum_{\chi \in \Irr(G)} \chi(t)\chi(c\alpha)\ov{\chi (d\alpha)} \ne 0.
    $$
    Take \(\rho \in \Gal(\mathbb{Q}_G \colon \mathbb{Q}_X)\) as in \cref{CoprimeAutChars}. Then
    \begin{eqnarray*}
        0 \ne \left(\sum_{\chi \in \Irr(G)} \chi(t)\chi(c\alpha)\ov{\chi (d\alpha)}\right)^\rho
            &= \sum_{\chi \in \Irr(G)} \chi^\rho(t)\chi^\rho(c\alpha)\ov{\chi^\rho (d\alpha)}   \\
            &= \sum_{\chi \in \Irr(G)} \chi(t)\chi(c\alpha^n)\ov{\chi (d\alpha^n)}.
    \end{eqnarray*}
    This establishes the claim.
\end{proof}

\begin{Lem} \label{lem:3r to 2r}
    Let $X=\SL_2(3^r)$ with $r$ an odd prime and $(\abs{X},r)=1$. Assume that $\tau \in \Aut(X)$ has order $3r$. Then there exists an involution $t \in \Inn(X)$, such that $t\tau$ has order $2r$.
\end{Lem}

\begin{proof}
    Set $\mathbb F= \GF(3^r)$. We calculate in $X=\SL_2(\mathbb F)$ and later take the central quotient. Let $\sigma$ be the automorphism of $\mathbb F$ which maps $w\in \mathbb F$ to $w^3$. Then $\sigma$ has order $r$ and determines a field automorphism of $X$.

    Since $\tau$ has order $3r$, we may choose an element $d=e\sigma^j$ of $X\gen{\sigma}$ such that $d$ projects to $\tau$ mod $Z(X)$. Furthermore, as $\tau$ has odd order and $Z(X)$ has order $2$, we may assume that $d$ has order $3r$. As $3$ and $r$ are coprime, we may assume $e$ and $\sigma^j$ commute. This yields that $e\in C_X(\sigma^j)=\SL_2(3)$. Assume that we can find $t$ such that $te\sigma^j$ has order $4r$ (as required), then for $f\in \GL_2(3) $ centralized by $\sigma^j$, we have $t^fe^f\sigma^j$ has order $4r$ and therefore we may as well assume that $e=\left(\begin{smallmatrix} 1 & 0 \\ 1 & 1 \end{smallmatrix}\right)$.

    Recall that for $M \in X=\SL_2(\mathbb F)$, the characteristic polynomial of $M$ is $x^2-\operatorname{Tr}(M)x+1$. Hence an element of $X$ has order $4$ if and only if it has trace $0$. To prove our result, we need to show that there exists an element $t\in X$ of order $4$ such that $td=te\sigma^j$ has order $4r$. We prove this first for $j=-1$. Then, setting $h= te$, we require $(h\sigma^{-1})^r$ to have trace $0$. This occurs if and only if
    $$
        \operatorname{Tr}(hh^{{\sigma}}\dots h^{{\sigma}^{r-1}})=0.
    $$

    Recall that we have $e=\left(\begin{smallmatrix} 1 & 0 \\ 1 & 1 \end{smallmatrix}\right) \in X$. We consider $t$ of the form $t=t(\lambda)= \left(\begin{smallmatrix} 0 & \lambda^{-2} \\ -\lambda^{2} & 0 \end{smallmatrix}\right)$ and note that $t$ has order $4$ and determinant $1$. Thus we have $\abs{\mathbb F}-1=3^r-1$ choices for $t$ and we aim to show that one of these choices results in $h=h(\lambda)= t(\lambda)e$ having $\operatorname{Tr}((h\sigma^{-1})^r)=0$.

    We can and do write $h\sigma^{-1}= \left(\begin{smallmatrix} \lambda^{-2} & 0 \\ 0 & \lambda^{-2} \end{smallmatrix}\right) \left(\begin{smallmatrix} 1 & 1 \\ -\lambda^{4} & 0 \end{smallmatrix}\right) \sigma^{-1}$ as
    $$
        h = te = \left(\begin{smallmatrix} \lambda^{-2} & \lambda^{-2} \\ -\lambda^2 & 0 \end{smallmatrix}\right) = \left(\begin{smallmatrix} \lambda^{-2} & 0\\ 0 & \lambda^{-2} \end{smallmatrix}\right) \left(\begin{smallmatrix} 1 & 1 \\ -\lambda^{4} & 0 \end{smallmatrix}\right).
    $$
    Then writing $z = \left(\begin{smallmatrix} 1 & 1 \\ -\lambda^{-4} & 0 \end{smallmatrix}\right),$ we calculate
    $$
        (h\sigma^{-1})^r= \left(\begin{smallmatrix} (\lambda^{-2})^{1+3+\dots + 3^{r-1}} & 0 \\ 0 & (\lambda^{-2})^{1+3+\dots + 3^{r-1}} \end{smallmatrix}\right) zz^\sigma\dots z^{\sigma^{r-1}} = zz^\sigma\dots z^{\sigma^{r-1}}.
    $$
    As $\operatorname{Tr}(zz^\sigma\dots z^{\sigma^{r-1}})^\sigma = \operatorname{Tr}( z^\sigma\dots z^{\sigma^{r-1}}z) = \operatorname{Tr}(zz^\sigma\dots z^{\sigma^{r-1}})$, (using that $\operatorname{Tr}(XY)=\operatorname{Tr}(YX)$ and $\operatorname{Tr}(X)^\sigma= \operatorname{Tr}(X^\sigma)$ for matrices $X$, $Y$), we obtain
    \begin{claim}   \label{trace good}
        $\operatorname{Tr}(zz^\sigma\dots z^{\sigma^{r-1}})$ is fixed by $\sigma$ and hence is in $\GF(3)$.
    \end{claim}

    \medskip

    We continue by considering
    $a = a(y) = \left(\begin{smallmatrix} 1 & 1 \\ y & 0 \end{smallmatrix}\right) \in \SL_2(\mathbb F[y])$ where for a moment we abuse notation and assume that $\sigma$ acts on the coefficients of the polynomials in $\mathbb F[y]$ and at the same time raises $y$ to the power $3$. We intend to investigate an upper bound for the degree of the entries of the matrix
    $$
        A_{k }(y)=aa^\sigma \dots a ^{\sigma^{2k}}
    $$
    for $k \ge 1$. Let $D_{k}(y)$ be the matrix which has $(\ell,m)$th entry the degree of the $(\ell,m)$th entry of $A_k(y)$. Set $c(y)= aa^\sigma = \left(\begin{smallmatrix} 1+y^3 & 1 \\ y & y \end{smallmatrix}\right)$. Then, using $A_{k}(y) = A_{k-1}(y) c(y)^{\sigma^{2k-1}} = A_{k-1}(y)c(y^{3^{2k-1}})$ and noting that $A_1(y) = a c(y)^\sigma$, induction shows that, for $k \ge 1$,
    $$
        D_{k}(y) = \left(
        \begin{matrix}\
            \sum_{t=1}^{k}3^{2t}    &
            \sum_{t=0}^{k-1}3^{2t+1}    \\
            \sum_{t=0}^{k}3^{2t}    &
            1+\sum_{t=1}^{k-1}3^{2t+1}
        \end{matrix}
        \right).
    $$
    We are required to consider the trace of $A_{\frac{r-1}{2}}(- x^{4})$, where $x$ is a variable taking values in $\mathbb F$. In particular, $x^{3^r}=x$. Notice that some of the degrees in $D_{\frac{r-1}{2}}(-y^4)$ are greater than $3^r$, but that those of $D_{\frac{r-3}{2}}(-y^4)$ are all less than $3^r$. Thus we may calculate the degrees in $D_{\frac{r-1}{2}}(-x^4)$ by multiplying $A_{\frac{r-3}{2}}(-x^4)$ by $c(-x^4)^{\sigma^{r-2}}$ with the degrees in $c(-x^4)^{\sigma^{r-2}}$ reduced by noting that $x^{3^r}=x$. Thus the degree matrix of $c(-x^4)^{\sigma^{r-2}}$ is $\left ( \begin{smallmatrix} (3+1)3^{r-1} & 0 \\ (3+1)3^{r-2} & (3+1)3^{r-2} \end{smallmatrix}\right)$ which reduces to $\left(\begin{smallmatrix} 3^{r-1} +1 & 0 \\ 3^{r-1}+3^{r-2} & 3^{r-1}+3^{r-2} \end{smallmatrix}\right)$. Assume that the degree matrix for $A_{\frac{r-1}{2}}(-x^4)$ has entries $d_{ij}$. We are only interested in $d_{11}$ and $d_{22}$ so we calculate
    \begin{align*}
        d_{11}  &= \max\{4\left(\sum_{t=1}^{\frac{r-3}{2}}3^{2t}\right)+3^{r-1}+1,4\left(\sum_{t=0}^{\frac{r-3}{2}-1} 3^{2t+1}\right)+ 3^{r-2}+3^{r-1}\}\\
                &= 4\left(\sum_{t=0}^{\frac{r-3}{2}-1}3^{2t+1}\right)+3^{r-2}+3^{r-1} = \sum_{s=1}^{r-1} 3^s= \frac{3^r-1}{2}-1
    \end{align*}
    and
    \[
        d_{22}  =  4 \left(1+\sum_{t=1}^{\frac{r-3}{2}-1}3^{2t+1}\right) +3^{r-2} + 3^{r-1}.
    \]
    We conclude that the trace of $A_{r-1}(-x^4)$ is a polynomial $T(x)$ of degree at most
    \[
        \max\{d_{11},d_{22}\} = \frac{3^r-1}{2}-1.
    \]
    Since $2\deg T(x) \le \abs{\mathbb F}-3$ and $T(x)$ evaluates in $\GF(3)$ by \ref{trace good}, the fact that $T(x)-\mu =0$ has at most $\deg T$ solutions, implies that $T$ takes all values in $\GF(3)$. In particular, there exists $\lambda \in \mathbb F\setminus \{0\}$ such that $T(\lambda)=0$. Hence, for this $\lambda$, $A_{\frac{r-1}{2}}(-\lambda^4)$ has trace $0$.

    Thus, writing $G= X\gen{\sigma}$, we have shown that $$(t^G)(e\sigma^{-1})^G \supseteq (h\sigma^{-1})^G$$ where $h\sigma^{-1}$ is $G$-conjugate to $w\sigma^{-1}$ with $w \in C_X(\sigma)$ of order $4$. Applying \cref{cor:one product enough}, yields
    $$
        (t^G)(e\sigma^{n})^G \supseteq (w\sigma^{n})^G
    $$
    for all $1\le n < r$. Thus, the selection of $j = -1$ is unimportant and there exists $t \in \Aut(X)$ of order $2$ such that $t\tau$ has order $2r$.
\end{proof}

\section{Properties of almost simple groups}\label{sec:simple}

We shall also need a rather large number of results about non-abelian simple groups. We start with a general result about the fixed points of elements in squares of conjugacy classes of the symmetric groups.

\begin{Lem} \label{lem:sym fix}
    Let \(n \geq 5\) and \(x \in \Sym(n)\) be such that \(\abs{\Fix(x)} = t < n-3\). Then there exists \(y \in \Alt(n)\) such that \(n>\abs{\Fix(x x^y)} \geq t+1\).
\end{Lem}

\begin{proof}
    Let \(x \in \Sym(n)\) with \(\abs{\Fix(x)} = t < n-3\) and \(x_1\) be a cycle involved in \(x\) of maximal length \(\ell\). If $\ell = 4$, then with $x_1= (d,e,f,g)$, $x_1x_1^{(d,e,f)}= (d,f,e)$ and so $\Fix(xx^{(d,e,f)}) = t+1$. For \(\ell \geq 5\), we let \(x_1 = (d,e,f,g,h\ldots)\) and \(y = (d,e)(f,g)\) and observe that \(x_1 (x_1)^y \) fixes $d$ and $f$ and is non-trivial. Hence \(\abs{\Fix(x x^y)} = t+2 \geq t+1\). We conclude that \(\ell \leq 3\).

    Now suppose \(\ell = 3\). If \(t \geq 2\) or \(x\) involves a transposition, we may assume that \(x_1 = (d,e,h)\) and \(x\) stabilises the set \(\{f,g\}\). Setting \(y = (d,e)(f,g)\) we see that \(x_1 (x_1)^y = 1\) and so \(\abs{\Fix(x x^y)} = t+3 \geq t+1\). Thus \(t < 2\), \(x\) is a product of 3-cycles and, as \(n \geq 5\), there are at least two of them, so we may assume \(x = (d,e,f)(g,h,i)\ldots\) and let \(y = (d,e)(g,h)\) to see that \(\abs{\Fix(x x^y)} = t+6 \geq t+1\). We are done, unless $n\in\{6,7\}$. For these cases we just observe that if $x=(1,2,3)(4,5,6)$ and $y=(3,4,5)$ then $xx^y$ is a $5$-cycle.

    This leaves us with the case \(\ell = 2\), so \(x\) is a product of at least two 2-cycles. If \(t > 0\), suppose that the permutation \(x_1 = (d,e)(f,h)\) is involved in \(x\) and that \(x\) fixes the point \(g\). Then, again setting \(y = (d,e)(f,g)\), we have that \(x_1 (x_1)^y = (f,g,h)(d)(e)\) and so \(\abs{\Fix(x x^y)} = t+1\). Thus \(t = 0\) and we may assume \(x\) involves some permutation \(x_1 = (d,e)(f,h)(g,i)\). Now observe that \(x_1 (x_1)^y = (f,g)(h,i)\) and so \(\abs{\Fix(x x^y)} = t+2 \geq t+1\). This completes the proof.
\end{proof}

\begin{Lem} \label{altfact}
    Suppose that $6\le n \le 9$ and $x \in \Alt(n)$ has odd order. Then there exists an involution $z \in \Alt(n)$ such that $zx$ has order divisible by $4$.
\end{Lem}

\begin{proof}
    This is readily checked using {\sc Magma} \cite{Magma}.
\end{proof}

We'll also use the following elementary lemma.

\begin{Lem} \label{Burnside}
    Suppose that $G$ does not have a normal $2$-complement. Then there exists $x,y\in G$ such that $[x,y]$ has even order. In particular, this holds for non-abelian simple groups.
\end{Lem}

\begin{proof}
    This is plainly true if $G$ has non-abelian Sylow $2$-subgroups and so we may assume they are abelian. Let $S \in \Syl_2(G)$. Then, by Burnside's Normal $p$-complement Theorem, $N_G(S)> C_G(S)\ge S$. Hence there exists $x\in N_G(S)\setminus C_G(S)$ and $y \in S$ such that $1\ne [x,y]\in S$. This proves the claim.
\end{proof}

An element of a Lie type group defined in characteristic $p$ is \emph{semisimple} if it has order coprime to $p$  and is \emph{regular semisimple} if its centralizer has order coprime to $p$.

Recall that a non-trivial element $z \in G$ is \emph{strongly real}, if and only if there exists an involution $t \in G$ such that $z^t=z^{-1}$.
\begin{Lem} \label{lem:strgrl}
    Let $G$ be a ﬁnite simple group of Lie type. Assume that $G$ is not one of $\PSL_n(q)$, $\PSU_n(q)$, $\PSO_{4n+2}^\pm(q)$, $\E_6(q)$ or ${}^2\!\E_6(q)$ where $q$ is a power of a prime $p$. Then every semisimple element of $G$ is strongly real.
\end{Lem}

\begin{proof}
    This is \cite[Lemma 10]{Galt}.
\end{proof}

\begin{Lem} \label{lem:odd centralizer inv sol}
    Suppose that $p$ is an odd prime and $G$ is a group with $N=F^*(G)$ a simple group of Lie type in characteristic $p$. Assume that $N$ has an involution $t$ such that $C_G(t)$ is soluble. Then $N$ is one of $\PSL_2(p^c)$ with $c\ge 1$, $\PSL_3(3)$, $\PSU_3(3)$, $\PSL_4(3)$, $\PSU_4(3)$, $\PSp_4(3)$, $\POmega_7(3)$, $\G_2(3)$, ${}^2\!\G_2(3)$, $\POmega_8^+(3)$.
\end{Lem}

\begin{proof}
    This can be deduced from \cite[Table 4.5.1]{GLS3}.
\end{proof}

\begin{Lem} \label{lem:solvpara}
	Suppose that $K$ is the derived group of a simple Lie-type group $L$ of rank at least $2$. Assume that $P \le L$ is a maximal parabolic subgroup of $L$. If $P$ is soluble, then $K$ is one of the following groups: $\PSL_3(2)$, $\PSL_3(3)$, $\PSL_4(2)$, $\PSL_4(3)$, $\PSU_4(2)$, $\PSU_4(3)$, $ \PSU_5(2)$, $ \PSp_4(2)'$, $ \PSp_4(3)$, $\PSp_6(2)$, $\PSp_6(3)$, $\POmega_7(3)$, $\POmega_8^+(2)$, $\POmega_8^+(3)$, $\G_2(2)'$, $\G_2(3)$, ${}^3\!\D_4(2)$, $ {}^3\!\D_4(3)$, ${}^2\!\F_4(2)'$.
\end{Lem}

\begin{proof}
	This can be extracted from \cite[Lemma 5.6]{Burness}. See also \cite[Lemma 2.6]{ParkerSaunders} for this trimmed list.
\end{proof}

\begin{Lem} \label{lem:sol max impossible}
    Suppose that $G$ is an almost simple group with $G/F^*(G)$ cyclic of odd order. Assume that $F^*(G)$ is the derived subgroup of a Lie-type group of rank at least $2$ with a soluble maximal parabolic subgroup or, if $F^*(G)$ is defined in odd characteristic, an involution centralizer which is soluble. Let $b \in G$ have odd order. Then there exists an involution $z \in G$ such that $bz$ has order divisible by $4$.
\end{Lem}

\begin{proof}
    Using \cref{lem:odd centralizer inv sol,lem:solvpara}, this has been verified by {\sc Magma} \cite{Magma}.
\end{proof}

In the next lemma, we point out a modest abuse of notation which we have adopted. If $N$ is the derived subgroup of a group $G$ of Lie type, we call $N$ a group of Lie type and a subgroup $P$ of $N$ a parabolic subgroup of $N$, provided there exists a parabolic subgroup $P^*$ of $G$ with $P=P^* \cap N$.

\begin{Lem} \label{lem:BT}
    Suppose that $p$ is a prime and $G$ is a group with $N=F^*(G)$ a simple group of Lie type in characteristic $p$. Assume that $x \in G$ normalizes a non-trivial $p$-subgroup of $N$. Then $x$ normalizes a parabolic subgroup of $N$.
\end{Lem}

\begin{proof}
    From among all non-trivial $p$-subgroups of $N$ which are normalized by $x$, choose $R$ of maximal order. Then $x$ normalizes $N_N(R)$ and $R \le O_p(N_N(R))$ with $O_p(N_N(R))$ normalized by $x$. Hence $R= O_p(N_N(R))$ and \cite[Corollary 3.1.5]{GLS3} implies that $N_N(R)$ is a parabolic subgroup of $N$ which is normalized by $x$.
\end{proof}

We will also use Gow's Theorem.

\begin{Thm}[Gow]\label{Gow'sThm}
Let $G$ be a ﬁnite simple group of Lie type of characteristic $p$, and let
$g$ be a non-identity semisimple element in $G$. Let $L_1$ and $L_2$ be conjugacy classes of
$G$ consisting of regular semisimple elements. Then $g \in L_1L_2$.
\end{Thm}
\begin{proof} This is \cite[Theorem 2]{Gow}. \end{proof}

We say that an element $x$ in a group $G$ is a \emph{projective involution} if $x\not \in Z(G)$ but $x^2\in Z(G)$.

\begin{Prop}    \label{cosets and involutions}
    Suppose that $G$ is a classical group of dimension $n \ge 2$ acting on its natural module $V$.
    \begin{enumerate}[(i)]
        \item If $F^*(G)= \SL_n(q)$ and $v,w\in V^\#$ with $\gen{v} \ne \gen{w}$ then there exists a projective involution $x$ such that $\gen{vx}=\gen{w}$. Furthermore, if either $q$ is even, $n > 2$ or $n=2$ and $G\ge \GL_2(q)$, then we can choose $x$ with $vx=w$.
        \item If $F^*(G)= \Sp_{2n}(q)$ and $v,w\in V^\#$ with $\gen{v} \ne \gen{w}$ then there exists a projective involution $x$ such that $\gen{vx}=\gen{w}$. Furthermore, if $q$ is even, $x$ can be chosen to be an involution with $vx=w$.
        \item If $F^*(G)= \SU_{n}(q)$ and $v,w\in V^\#$ are isotropic with $\gen{v} \ne \gen{w}$ then there exists a projective involution $x$ such that $\gen{vx}=\gen{w}$. Furthermore, if $q$ is even, $n > 2$ or $n=2$ and $G\ge \GU_2(q)$, then we can choose $x$ to be an involution with $vx=w$.
        \item If $F^*(G)= \Omega^\epsilon_{n}(q)$ with $n \ge 3$, $\epsilon =\pm$ when $n$ is even and $v,w\in V^\#$ are singular with $\gen{v} \ne \gen{w}$ then there exists an involution $x$ such that $\gen{vx}=\gen{w}$.
    \end{enumerate}
\end{Prop}

\begin{proof}
    (i) By hypothesis, $v$ and $w$ are linearly independent. Let $W=\gen{v,w}$ and $U=\gen{u_3,\dots , u_{n}}$ be a complement to $W$ in $V$. Then consider the basis $\{v,w,u_3,\dots, u_n\}$ of \(V\) and write elements of $G$ with respect to this basis.

    If $n=2$, $\begin{psmallmatrix} 0 & 1 \\ -1 & 0 \end{psmallmatrix} \in \SL_2(q)$ is a projective involution which exchanges $\gen{v}$ and $\gen{w}$ whereas, if $n\ge 3$, we may exchange $v$ and $w$, negate $u_3$ and fix $u_4, \dots, u_n$ to obtain an involution $x$ satisfying $vx=w$. When $n= 2$ and $G$ contains $\GL_2(q)$, we can swap $v$ and $w$.

    (ii) We have $\gen{v}$ and $\gen{w}$ are totally isotropic. Let $W=\gen{v,w}$ and $U= W^\perp$. Then either $W$ is totally isotropic or $W$ is non-degenerate. If $W$ is totally isotropic, then $n \ge 2$ and $\Stab_G(W)$ induces $\GL_2(q)$ on $W$. The result then follows from (i). So assume that $W$ is non-degenerate. Then $\Stab_G(W)= \Sp(W) \times \Sp(U)$. By (i), we can choose a projective involution $x\in \Sp(W)$ to exchange $\gen{v} $ and $\gen{w}$. In $\Sp(U)$ we take a projective involution $y$. Then $xy$ is a projective involution demonstrating the asserted result.

    (iii) We are given that $\gen{v}$ and $\gen{w}$ are totally isotropic. Let $W=\gen{v,w}$ and $U= W^\perp$. Then either $W$ is totally isotropic or $W$ is non-degenerate. If $W$ is totally isotropic, then $n \ge 4$ and $\Stab_G(W)$ induces $\GL_2(q^2)$ on $W$. The result then follows from (i). So assume that $W$ is non-degenerate. Then $\Stab_G(W)= \GU(W) \times \GU(U)$.  If $U= W$, then we note that with respect to the basis $\{v,w\}$, both $\begin{psmallmatrix} 0&1\\-1&0\end{psmallmatrix}$ and $\begin{psmallmatrix} 0&1\\1&0\end{psmallmatrix}$ are in $\GU(W)$. Noting that we can negate a non-isotropic vector of $U$ and fix its perpendicular space element-wise within $\GU(U)$, we have a proof of (iii).

    (iv) Let \(W = \gen{v, w}\). Again either \(W\) is totally isotropic or non-degenerate. If the former holds, we have \(n \geq 4\), \(\Stab_G(W)\) induces \(\GL_2(q)\) on \(W\) and the result follows from (i). Thus suppose \(W\) is non-degenerate. If \(q\) is odd we may take \(x\) to be any involution which swaps \(v\) and \(w\) and negates some non-singular vector \(u\) in \(W^{\perp}\), scaling \(u\) and \(v\) if needed to ensure that \(x\) lies in \(\Omega\) (see \cite[p29--30, Descriptions 1 \& 2]{KleidmanLiebeck}). If \(q\) is even and \(G\) is not \(\Omega_4^-(q)\), we may take \(x\) to be the involution which swaps \(v\) and \(w\) as well as swapping two isometric 1-spaces in \(W^{\perp}\). If \(G = \Omega_4^-(q)\), then we $G\cong \PSL_2(q^2)$ and $\Stab_G(\gen{v})$ is a Borel subgroup of $G$. By (i), every coset of $yB\ne B$ in $G$ contains an involution. It follows that $\gen{v}$ and $\gen{w}$ are swapped by an involution.
\end{proof}

\begin{Cor} \label{Cornice}
    Suppose that $G$ is a simple projective classical group and $P$ is a maximal subgroup of $G$ stabilizing a totally isotropic $1$-space. Then every coset of $P$, other than perhaps $P$, contains an involution.
\end{Cor}

\begin{proof}
    By \cref{cosets and involutions}, every coset of $P$, other than perhaps $P$, contains the image of a projective involution. Since the image in $G$ of a projective involution is an involution. This proves the claim.
\end{proof}

Of course, $P$ itself is of odd order if and only if $P$ is a Borel subgroup of $\PSL_2(q)$ where $q$ and $(q-1)/2$ are odd.

Finally, we present a result about sporadic simple groups that was obtained by computer.

\begin{Lem} \label{spor-calcs}
    Suppose that $G$ is a sporadic simple group and let $C$ be a conjugacy class of $G$ of odd-order elements. Then
    \begin{enumerate}[(i)]
        \item $C^2$ contains an element of even order; and
        \item either there exists an involution $t \in G$ such that $tC$ contains an element of order $4$, or $G \cong \J_1$.
    \end{enumerate}
\end{Lem}

\begin{proof}
    These facts are checked using the character tables of the sporadic simple groups, structure constants and {\sc GAP} \cite{GAP}.
\end{proof}

\section{Initial results for the proof of Theorem~\ref{C3}}\label{sec:red}

In this section, we begin the proof of \cref{C3}. Suppose that $(G,K)$ is a counterexample to \cref{C3} with $\abs{G}$ minimal. Thus $K$ is a normal subset of $G$ consisting of elements of odd order, $K^2 \subseteq \mathbf D_K$ and $\gen{K}$ is not soluble. We recall here that
$$
    \mathbf D_K=\{y\in G\mid \gen{y}=\gen{x} \text{ for some } x \in K\}.
$$

Our objective in this section is to deduce some structural information about the minimal counterexample $(G,K)$. We begin with some general observations which flow from the hypothesis of \cref{C3}.

\begin{Lem} \label{lem:autos}
    Assume that $\alpha \in \Aut(G)$. Then $(G,K\alpha)$ is also a minimal counterexample.
\end{Lem}
\begin{proof}
    This is obvious.
\end{proof}

\begin{Lem} \label{lem: 2 for now}
    The set of orders of elements of $K$ is the same as the set of element orders of elements in $\mathbf D_K$. In particular, $\mathbf D_K$ consists of odd order elements.
\end{Lem}

\begin{proof}
    By definition $K \subseteq \mathbf D_K$. So the set of orders of elements of $K$ is a subset of the set of orders of elements in $\mathbf D_K$. On the other hand, if $y \in \mathbf D_K$, then there exists $x \in K$ such that $y \in \mathbf D_x$. But then $x$ and $y$ have the same order. This completes the proof.
\end{proof}

The main inductive tool is provided by the next lemma.

\begin{Lem} \label{lem:subgrps ok}
    If $H < G$, then $\gen{H\cap K}$ is a soluble normal subgroup of $H$.
\end{Lem}

\begin{proof}
    We may as well assume that $H \cap K \not=\emptyset$. Note that $H \cap K$ is a normal subset of $H$. Suppose that $x, y \in H \cap K$. Then $xy \in K^2 \subseteq  \mathbf D_K$. Hence there exists $z \in K$ such that  $\gen {z} =\gen{ xy} \le H$. Thus $z \in H \cap K$ and consequently $(H\cap K)^2\subseteq \mathbf D_{H \cap K}$. The minimality of $G$ yields $\gen{H\cap K}$ is a soluble normal subgroup of $H$.
\end{proof}

One consequence of \cref{lem:subgrps ok} and $G$ being a minimal counterexample to \cref{C3} is that
$$
    G= \gen{K}
$$
is not soluble.

The following lemma requires that we allow the identity of $G$ to be in $K$.

\begin{Lem} \label{lem:quotients}
    Suppose that $\theta \colon G \rightarrow H$ is a homomorphism with $\ker \theta \ne 1$. Then $\mathrm{Im} (\theta)=\gen{\theta(K)}$ is soluble.
\end{Lem}

\begin{proof}
    We may as well suppose that $\theta$ is an epimorphism. Then $\theta (K) =\{\theta(x)\mid x \in K\}$ is a normal subset of $\theta(G)=H$ and every element of $\theta(K)$ has odd order. Let $\theta(x), \theta(y) \in \theta(K)$. Then, as $K^2 \subseteq \mathbf D_K$, there exists $z \in K$ such that $\gen{xy}= \gen{z}$. Hence $\gen{\theta(x)\theta(y)} = \gen{\theta(z)}$. This yields $$\theta(x) \theta(y) \in \mathbf D_{\theta(z)}\subseteq \mathbf D_{\theta(K)}.$$ Therefore $\theta(K) ^2 \subseteq \mathbf D_{\theta(K)}$. Hence $\mathrm{Im}(\theta)= \gen{\theta(K)}$ is soluble as $\ker \theta \ne 1$ and $(G,K)$ is a minimal counterexample.
\end{proof}

\begin{Cor} \label{cor:normal ok}
    Suppose that $N$ is a non-trivial normal subgroup of $G$. Then $G/N$ is soluble and $N$ is not.
\end{Cor}

\begin{proof}
    This follows from \cref{lem:quotients} as $G$ is not soluble.
\end{proof}

\begin{Lem} \label{lem:min normal}
    There is a unique minimal normal subgroup $N$ of $G$. Furthermore, $C_G(N)=1$.
\end{Lem}

\begin{proof}
    Suppose that $M$ and $N$ are distinct minimal normal subgroups of $G$. By \cref{cor:normal ok}, both $M$ and $N$ are not soluble and $G/N$ is soluble. Since $M\cong MN/N \le G/N$ is soluble, we have a contradiction. Hence $G$ has exactly one minimal normal subgroup. Calling this minimal normal subgroup $N$, we know that $N$ is not abelian by \cref{cor:normal ok}. Hence the normal subgroup $C_G(N) $ does not contain $N$ and consequently $C_G(N)=1$.
\end{proof}

We shall use the following notation until the proof of \cref{C3} is complete.

\begin{Not} \label{notation}
    Let $N$ be the unique minimal normal subgroup of $G$. Write $$N = N_1 \times \dots \times N_n$$ where $n \ge 1$ and define $\pi_1$ to be the projection from $N$ onto $N_1$.
\end{Not}

Notice that $N_i$, $1 \le i \le n$ is a non-abelian simple group and that these subgroups are pairwise isomorphic to each other.

\begin{Lem} \label{lem:int}
    Either $G= N$ is a simple group or $(K\cap N)^\#=\emptyset$.
\end{Lem}

\begin{proof}
    If $G= N$ then, as $G$ has a unique minimal normal subgroup, $n=1$ and $G$ is a simple group.

    Assume that $N < G$ and $(K\cap N)^\#\ne\emptyset$. Then $(K \cap N)^\#$ is a non-empty normal subset of $G$ and so $\gen{K \cap N} = N <G$ as $N$ is a minimal normal subgroup of $G$. This of course contradicts the combination of \cref{lem:subgrps ok,cor:normal ok}.
\end{proof}

The next result is our main result about the structure of a minimal counterexample to Theorem~\ref{C3}.

\begin{Thm} \label{lem:G/N cyclic}
    Either $G= N$ is a non-abelian simple group, or $(K \cap N)^\# =\emptyset$ and $G= N \gen{a}$ for all $a \in K^\#$. In particular, $G/N$ is cyclic of odd order.
\end{Thm}

\begin{proof}
    Assume that $G>N$. We know that $(K \cap N)^\#=\emptyset$ by \cref{lem:int}. Let $a \in K^\#$, put $H= N \gen{a}$ and assume that $H < G$. Then $\gen{K \cap H}$ is a non-trivial soluble normal subgroup of $H$ by \cref{lem:subgrps ok}. Thus $[\gen{K \cap H},N]$ is a soluble normal subgroup of $N$ and hence is the trivial subgroup. As $a\in K \cap H$, we have $a \in C_G(N)$. Since $a\ne 1$, this contradicts \cref{lem:min normal}. Therefore, for $a \in K^\#$, $G= N \gen{a}$. Since every element of $K$ has odd order, $G/N$ is cyclic of odd order.
\end{proof}

One consequence of \cref{lem:G/N cyclic} and $N$ being a minimal normal subgroup of $G$ is that, for $a \in K^\#$, $\gen{a}$ acts transitively on $\{N_1, \dots,N_n\}$ by conjugation. In particular, $a^n$ normalizes each $N_j$, $1\le j \le n$ though we should be aware that $a^n$ might not be an element of $N$.

The next corollary relies on the fact that in a finite simple group there are commutators of even order.

\begin{Cor}\label{cor:wreathgood}
    If \([G : N] = n \ge 2\), then \(K\) contains no elements of order \(n\).
\end{Cor}

\begin{proof}
    Suppose \(a \in K\) has order \(n\). Then $G=N\gen{a}$ by \cref{lem:G/N cyclic} and $G \cong N_1\wr \gen{a}$. As \(N_1\) is a non-abelian simple group, using \cref{Burnside} we can find \(x\), \(y\in N_1\) such that \([x,y]\) has even order. Take \(w\) as in the statement of \cref{lem:wreath orders}. Then \(a^w a \in K^2\subseteq \mathrm{D}_K\) has even order, contradicting \cref{lem: 2 for now} and the definition of \(K\).
\end{proof}

\begin{Lem} \label{dividers}
    Suppose that $n \ge 2$ and $a\in K^\#$. If $r$ is a prime divisor of $n$, then $r$ divides the order of $a^n$. In particular, $K$ and $K^2$ have no elements of order $n$.
\end{Lem}

\begin{proof}
    Suppose that $r$ does not divide the order of $a^n$ and assume that $a$ has order $m$.

    Then $y=a^{m/r}$ is an element of order $r$ and $\gen{y}$ acts semiregularly on $\{N_1, \dots, N_n\}$. It follows that $C_N(y)$ is a direct product of $n/r$ subgroups of $N$ each isomorphic to $N_1$ and $\pi_1(C_N(y))=N_1$. Since $a$ normalizes $C_N(y)$, $\gen{a^{C_G(y)}}$ is a soluble group by \cref{lem:subgrps ok}. Thus
    $$
        [\gen{a^{C_G(y)}}, C_N(y)] \le \gen{a^{C_G(y)}}\cap C_N(y)=1
    $$
    as $C_N(y)$ has no non-trivial soluble normal subgroups. Hence $a\in \gen{a^{C_G(y)}}$ centralizes $C_N(y)$ and so $C_N(y)= C_N(a)$.

    As $\gen{a}$ acts transitively on $\{N_1, \dots, N_n\}$, we deduce that $n=r$. Now $a^r$ normalizes $N_i$ for $1\le i \le r$, and so
    $$
        C_N(y)=C_{N}(a) \le C_N(a^r)=C_{N_1}(a^r)\dots C_{N_r}(a^r).
    $$
    As
    $$
        N_1=\pi_1(C_N(y))\le \pi _1 (C_{N_1}(a^r)\dots C_{N_r}(a^r))= C_{N_1}(a^r),
    $$
    we conclude that $a^r$ centralizes $N_1$ and so also centralize $N$. Thus $a^r=1$ by \cref{lem:min normal}, and consequently $a=y$ has order $r$. In particular, $\abs{G:N}=r=n$ and so we apply \cref{cor:wreathgood} to obtain a contradiction.
\end{proof}

\begin{Lem} \label{lem:sol over}
    Suppose that $a \in K^\#$ and $n \ge 2$. Then, either $a^n=1$ or every proper subgroup of $N_1$ which is normalized by $a^n$ is soluble. In particular, for $k \ge 1$, either $a^{nk}=1$ or $C_{N_1}(a^{nk})$ is soluble.
\end{Lem}

\begin{proof}
    Assume that the lemma is false and choose a subgroup $L_1$ of $N_1$ of minimal order exhibiting this fact. Then $L_1$ is perfect and non-trivial with $L_1\ne N_1$. For $2\le i \le n$, define $L_i=L_1^{a^{i-1}}$. Then $a$ normalizes $X=L_1 \dots L_n $ and $H=X\gen{a} < G$. By \cref{lem:subgrps ok}, $\gen{a^X} \le \gen{H \cap K}$ is a soluble normal subgroup of $H$ and, furthermore,
    $$
        [X,a] \le X \cap \gen{a^X}
    $$
    is soluble.

    For $1\le i \le n$, define $M_i$ to be the smallest characteristic subgroup of $L_i$ which has $L_i/M_i$ a direct product of non-abelian simple groups.

    Set $M= M_1\dots M_n$. Then $M$ is normal in $X$ and $[X,a]M/M$ is a soluble normal subgroup of $X/M$ which is a direct product of non-abelian simple groups. Therefore $[X,a] \le M$. Now let $x \in L_1 \setminus M_1$. Then $x^a \in L_2 \setminus M_2$ and so $[x,a]= x^{-1}x^a\not \in M$, a contradiction as $[x,a]\in [X,a] \le M$. This proves the claim.
 \end{proof}

We finish this section with some observations related to finding elements of $K^\#$ in the centralizer of an involution.

\begin{Lem}\label{lem:inv1}
Suppose that $t \in G$ is an involution and $a\in K^\#$.  If $at$ is an involution then $G$ is simple.
\end{Lem}

\begin{proof} Using  \cref{lem:G/N cyclic} we know $G=\gen{a}N$ and $G/N$ is cyclic of odd order. In particular, $t \in N$ and
   $$N=atatN=a^2[a,t]N=a^2N.$$
As $\gen{a^2}=\gen{a}$, we deduce that $a \in N$ and $G=N$ is a simple group.
\end{proof}

\begin{Lem} \label{lem: invs}
    Suppose that $t\in G$ is an involution and $a \in K$. Then $\{ aa^t,a^ta\} \subseteq \mathbf D_K$ and  either $\{at, ta \}\subseteq \mathbf D_K$ or $at$ and $ta$ have even order. In particular, $ta$ either has odd order or twice odd order for all $a \in K$ and involutions $t$ in $G$.
\end{Lem}

\begin{proof} Our main hypothesis asserts that $\{ aa^t,a^ta\} \subseteq K^2 \subseteq \mathbf D_K$.
We consider $at$, noticing that $ta=(at)^t$. Since $$(at)^2=atat= aa^t   \in \mathbf D_K,$$ there exists $a_* \in K$ such that
    $$
        \gen{a_*} =\gen{aa^t }= \gen{(at)^2} \le \gen{at}.
    $$
    Since $\abs{\gen{at}:\gen{(at)^2}} \le 2$, we either have $\gen{a_*}=\gen{(at)^2}=\gen{at}$, which means that $at\in \mathbf D_K$, or $\gen{at}$ has even order. Since $(at)^2$ has odd order, this proves the claim.
\end{proof}

\begin{Lem}\label{lem:invs3}
  Suppose that $s\in G$ is strongly real and $a\in (K\cap N_G(\gen{s}))^\#$. Let $w\in N_G(\gen{s})$ be an involution inverting $s$.  Then either $aw$ is an involution and $G$ is a simple group or there exists an involution $t \in G$ such that $(K\cap C_G(t))^\#$ is non-empty.
\end{Lem}

\begin{proof}
  Let $C= \{y\in G\mid s^y = s^{\pm 1}\}$. Then $C$ is a subgroup of $G$ and, as $s$ is real, $\abs{C:C_G(s)}= 2$. Furthermore, $N_G(\gen{s})$ normalizes $C$ and centralizes $C/C_G(s)$. By hypothesis, $w\in C$  inverts $s$ and we have $a \in (K \cap N_G(\gen{s}))^\#$. Since $a$ has odd order and $wC_G(s)\in Z(N_G(\gen{s})/C_G(s))$   is centralized by $aC_G(s)$, $aw$ has even order. Let $t \in \gen{aw}$ have order $2$.
  By \cref{lem: invs}, $(aw)^2\in \mathbf D_K$ and so there exists $a_*\in K$ such that $\gen{a_*}= \gen {(aw)^2} \le C_G(t)$. We have two possibilities, either $a_*=1$ or $a_*\in (K \cap C_G(t))^\#$. We may assume that the former holds.   Then \cref{lem:inv1} applied to $aw$ implies that $G$ is simple. This proves the result.\end{proof}

For the remainder of the paper, we investigate further a minimal counterexample $G$ to \cref{C3} and maintain the notation established in this section. Thus $a$ will always be an element of $K^\#$, $G= N \gen{a}$ and $$N= N_1 \times \dots \times N_n$$ is the unique minimal normal subgroup of $G$. We also write $b=a^n$. The next three sections investigate the various possibilities for the subgroup $N_1$.

\section{Alternating and sporadic simple groups}\label{sec:Alt}

In this section, we show that $N_1$ cannot be an alternating group or a sporadic simple group.

\begin{Lem} \label{lem:alt wr gone}
    Suppose that $n> 1$. Then $N_1 \not \cong \Alt(m)$ for $m \ge 5$.
\end{Lem}

\begin{proof}
    Write $\Omega=\{1, \dots, m\}$. Since $\Out(\Alt(m))$ is a $2$-group, $b=a^n\in N$. Let $b_1=\pi_1(b)$. If $b_1$ is not an $m$-cycle, then $b_1$ stabilizes the non-empty subsets $\Omega_1$ and $\Omega_2$ of $\Omega$ where $\Omega= \Omega_1\cup \Omega_2$ is a disjoint union. It follows that the subgroup of $N_1$ which leaves $\Omega_1$ and $\Omega_2$ invariant is normalized by $\gen {b}$. It follows that $\Alt(\Omega_1)$ and $\Alt(\Omega_2)$ are both soluble by \cref{lem:sol over}. Therefore $m \le 8$. Then by combining \cref{prod inv order,altfact}, we have $m=5$. Hence $b_1$ is a $3$-cycle. Now we may choose an involution $z$ to invert $b_1$ and thus $aa^z$ has order $n$ by \cref{prod inv order}, contrary to \cref{dividers}. We have proved that $b_1$ is an $m$-cycle and in particular we know that $m$ is odd. Let $k$ be a divisor of $m$. Then $b_1$ has overgroup  isomorphic to $\Alt(k)\wr \Alt(m/k)$    which is normalized by \(\gen{b}\).  Thus \cref{lem:sol over} implies that $m/k$ and $k$ are both at most $3$.
    Hence either $m$ is prime, or $m = 9$. The latter case is impossible by \cref{prod inv order} and \cref{dividers}. Hence $m$ is a prime and $n$ is a power of $m$ by \cref{dividers}. We may suppose that $b_1=(1,2,\dots, m)$. Now take $z=(1,2)(3,4)$. Then $b_1z$ fixes $1$ and $3$ and cyclically permutes the remaining elements of $\Omega$. Hence $b_1z$ has order $m-2 $ and $az$ has order $n(m-2)$ by \cref{prod inv order}. Now, as $n$ and $m-2$ are odd, $aa^z = (az)^2$ has order $n(m-2)$. Taking $a_* \in K$ so that $\gen{a_*}= \gen{aa^z}$ and applying \cref{dividers} to $a_*$, we have a contradiction.
\end{proof}

\begin{Lem} \label{lem:Alt Gone}
    We have $N_1 \ne \Alt(m)$ for $m \ge 5$.
\end{Lem}

\begin{proof}
    By \cref{lem:alt wr gone}, we have $G=N=N_1$. Choose $a \in K^\#$ with $\abs{\Fix(a)}$ maximal. Let $t=\abs{\Fix(a)}$. If $a$ is not a $3$-cycle, then $t < n-3$ and so \cref{lem:sym fix} implies that there exists $y \in G$ such $\abs{\Fix(aa^y)} > t$ and $aa^y \ne 1$. By hypothesis there exists $z \in K$ such that $\gen{z} = \gen{aa^y}$. But then $\abs{\Fix(z)}> t$, a contradiction. Hence $a$ is a $3$-cycle. Now we simply note that $(1,2,3)(2,3,4)$ has order $2$ and have a contradiction.
\end{proof}

\begin{Lem} \label{notspor}
    We have $N_1$ is not a sporadic simple group.
\end{Lem}

\begin{proof}
    Let $a\in K$. Then $G=N\gen{a}$ and as the outer automorphism groups of the sporadic simple groups have order dividing $2$, $b=a^n \in N$.

    If $N= N_1$, then $a\in N=G$. As $a^G$ consists of odd-order elements, \cref{spor-calcs} (i) implies $(a^G)^2\subseteq K^2 \not \subseteq \mathbf D_K$, which is a contradiction. Hence $n> 1$.

    If $N_1$ is not $\J_1$, then using \cref{spor-calcs} (ii) there exists an involution $z\in N_1$ such that $\pi_1(b)z$ has order divisible by $4$. But then $bzC_G(N_1) \in N/C_N(N_1) $ has order divisible by $4$ and \cref{prod inv order} implies that $az$ has order divisible by $4$. Now \cref{lem: invs} provides a contradiction. Thus $N_1 \cong \J_1$ and $n>1$.

    The odd-order elements of $J_1$ are strongly real. Hence, as $n>1$,  \cref{lem:invs3} implies that there exists an involution $t\in G $ such that $(K\cap C_G(t))^\#$ is non-empty.   So suppose that $a_*\in (K \cap C_G(t))$ and set $b_*=a_*^n$. Then $ C_{N_1}(\pi_1(t)) \cong 2\times \Alt(5)$. Hence $\pi_1(b_*) \in C_{N_1}(\pi_1(t))$ and $b_*$ normalizes $C_{N_1}(\pi_1(t))$. This contradicts \cref{lem:sol over}. We have shown that $N_1$ is not a sporadic simple group.
\end{proof}

\section{Lie-type groups of rank at least 2}    \label{sec:rank2}

We continue to assume that $G= N\gen{a}$, for non-trivial $a \in K$ is a minimal counterexample to \cref{C3}. In this section we assume further that \(N_1\) is a simple group of Lie type of Lie rank at least 2 which is defined in characteristic $p$ over the field $\GF(p^c)$ or the derived subgroup of such a group. We follow \cite[Definition 2.5.13]{GLS3} for our definitions of the various types of automorphisms of groups of Lie type. We start with the following lemma which eliminates most small cases.

\begin{Lem} \label{notsolmax}
    The following hold:
    \begin{enumerate}[(i)]
        \item all the maximal parabolic subgroups of $N_1$ are not soluble; and
        \item if $p$ is odd, all the involution centralizers in $N_1$ are not soluble.
    \end{enumerate}
\end{Lem}

\begin{proof}
    Suppose the lemma is false. Then, according to \cref{lem:sol max impossible,prod inv order}, we can find an involution $t\in N_1$ so that $ta$ has order divisible by $4$. This contradicts \cref{lem: invs}.
\end{proof}

\begin{Prop}    \label{prop:not plocal}
    No element of $K^\#$ normalizes a non-trivial $p$-subgroup of $N$.
\end{Prop}

\begin{proof}
    Assume that $a\in K^\#$ normalizes a $p$-subgroup $T$ in $N$. Set $b=a^n$. Since $b$ normalizes $N_1$ and $T$, $b $ normalizes $T_1=\pi_1(T)$. As $\gen{a}$ acts transitively on $\{N_1, \dots, N_n\}$, $T_1\ne 1$. Since \(b\) normalizes \(T_1\), \cref{lem:BT} yields that $b$ normalizes a parabolic subgroup $P$ of $N_1$. In particular, the conjugation action of \(b\) permutes the maximal parabolic subgroups of \(N_1\) containing \(P\). By \cref{notsolmax}, the maximal parabolic subgroups of $N_1$ are not soluble.

    Assume that $n \ne 1$. Then \cref{lem:sol over} yields $P$ and the parabolic subgroups of $N_1$ which contain $P$ and are normalized by $b$ are all soluble. Therefore, the maximal parabolic subgroups containing $P$, are not normalized by $b$. Hence, as $b$ has odd order, $N_1\gen{b}$ involves a triality automorphism and $N_1 \cong \POmega_{8}^+(p^c)$.   \cref{notsolmax} implies $p^c> 3$ and all the minimal parabolic subgroups of $N_1$ are non-soluble. Hence $P$ is a Borel subgroup of $N_1$. But then $P$ is contained in a maximal parabolic subgroup which is normalized by $b$ and such groups are non-soluble, a contradiction.   We conclude that $N=N_1$ is simple and that $b=a$.

    Choose a proper parabolic subgroup $Q$ of $N$ such that $\abs{K \cap N_G(Q)}$ is maximal and then that $ Q $ is   maximal by inclusion. Since $a \in K \cap N_G(P)$, $K \cap N_G(Q)$ is non-empty.

    Let $B \le Q$ be a Borel subgroup contained in $Q$. By \cref{lem:subgrps ok}, $A= \gen{K \cap N_G(Q)} $ is a soluble normal subgroup of $N_G(Q)$. In particular, $X=AB \cap N \ge B$ is a soluble parabolic subgroup of $G$ contained in $Q$ and containing $B$. By \cref{notsolmax}, $X$ is not a maximal parabolic subgroup of $N$. Because $N$ has Lie rank at least $2$, there exists a parabolic subgroup $R \ge X$ with $R \not\le Q$. Notice that $AB$ normalizes $X$ and, as $B \le Q$ and $A \le N_G(Q)$, $AB$ normalizes $Q$. Further, \(AB\) permutes the maximal parabolic subgroups of \(N\) which contain \(X \geq B\) by conjugation. Assume that $x \in AB=BA$ does not normalize $R$. Then $x= \beta\alpha $ for some $\alpha \in A$ and $\beta \in B$. Hence, as $B \le R$, we have $R^x=R^\alpha \ne R$. Write $\alpha = gfdi$ where $g$ is a graph automorphism, $f$ is a field automorphism, $d$ is a diagonal automorphism all chosen with respect to the Borel subgroup $B$ and $i$ is an inner automorphism (see \cite[Theorem 2.5.1]{GLS3}). Assume that $g=1$. Then \(df\) normalizes \(R\) and so
    $$
        R^\alpha =R^{gdfi}= R^i
    $$
    and $R^i$ is a parabolic subgroup containing $X$. Since no two distinct parabolic subgroups containing $B$ are $N$-conjugate, we deduce that $R^\alpha= R$. In particular, $AB \le N_G(R)$.

    On the other hand, if $g \ne 1$, then $g$ must be the triality automorphism and $N_1 \cong \POmega_{8}^+(p^c)$. By \cref{notsolmax} we have $p^c> 3$ and, as $X$ is soluble, we have $B=X$. Therefore, in our initial selection we may choose $Q$  to be $\gen{g}$-invariant and to correspond to the end nodes of the $\D_4$ Dynkin diagram. Now, in this case, we may choose $R \ge X=B$ corresponding to the middle node such that $AB \le N_G(R)$.

    Since $K \cap N_G(Q) \subseteq K \cap AB \subseteq K \cap N_G(R) $, the maximal choice of $Q$ yields $K \cap N_G(Q) =K \cap N_G(R) $. Hence $A=\gen{K \cap N_G(Q) }=\gen{K \cap N_G(R)}$ is normal in both $N_G(Q)$ and $N_G(R)$. Let $Y=\gen{Q,R}>Q$. Since $Q$ is maximal in $N$, $N=Y$ and $G= \gen{N_G(Q),N_G(R)}$. This implies that $A$ is normal in $G$, which is a contradiction as $A \ne 1$ is soluble. This completes the proof.
\end{proof}

\begin{Lem}\label{lem:GneN}
We have $G \ne N$. In particular, for all $a\in K^\#$, and involutions $t \in N$, the element $at$ does not have order $2$.
\end{Lem}

\begin{proof} Suppose that $G=N$ and let $a\in K^\#$. Then \cref{prop:not plocal} implies that $a$ is a regular semisimple element of $G$. By \cref{Gow'sThm}, $K^2$ contains every semisimple element of $G$. In particular, $\mathbf D_K$ contains every semisimple element of $G$.  Let $P$ be a maximal parabolic subgroup of $G$. As $G$ has Lie rank at least $2$, $P$ contains a semisimple element and so $(K \cap P)^\#$ is non-empty.  This contradicts \cref{prop:not plocal} as $O_p(P)\ne 1$. Hence $G= N\gen{a}> N$ for all $a \in K^\#$. If $t$ is an involution and $at$ has order $2$, then $a\in \gen{a}=\gen{t,ta}'\le N$.
\end{proof}

\begin{Lem} \label{CNa Odd}
    We have $\abs{C_N(a)}$ is odd for all $a\in K^\#$. Furthermore, setting $b= a^n$, $\abs{C_{N_1}(b)}$ is odd.
\end{Lem}

\begin{proof}
    Suppose that $a\in K^\#$ centralizes an involution $t$. Then $p$ is odd by \cref{prop:not plocal}.

    Then, as $a$ centralizes $t$, $b$ centralizes $t_1 = \pi_1(t)$ which is also an involution. By \cref{notsolmax}, $C_{N_1}(t_1)$ is not soluble. Since $b $ normalizes $C_{N_1}(t_1)$,   \cref{lem:sol over} implies that $n=1$ and so $N$ is simple.

    As $t \in C_N(a)$, $\gen{K \cap C_G(t)}$ is a non-trivial soluble normal subgroup of \(C_G(t)\). Hence $\gen{K \cap C_G(t)}$ centralizes $E(C_G(t))$ which, as $C_G(t)$ is not soluble, is non-trivial by \cite[Theorem 4.2.2]{GLS3}. Since  $E(C_G(t))$  contains elements of order $p$, this contradicts \cref{prop:not plocal}.   Finally note that if $b$ commutes with an involution $t$ in $N_1$  then $a$ commutes with the involution $\prod_{i=1}^n t^{a^i}$, which is impossible.
\end{proof}

\begin{Lem} \label{no field autos rk2}
    Let $a\in K^\#$ and $b=a^n$. Then no element of $\gen{b}$ of prime order induces by conjugation on $N_1$ an element in the coset of a field, graph-field or graph automorphism in $\Aut(N_1)/\Inn(N_1)$.
\end{Lem}

\begin{proof}
    Assume that $f \in \gen{b}$ of prime order $r$ acts on $N_1$ by conjugation as an element in the coset of a field, graph-field or graph automorphism. Since $a$ has odd order, in the latter two cases, $f$ has order $3$. Put $X=O^{p'}(C_{N_1}(f))$. In the first two cases, \cite[Proposition 4.9.1]{GLS3} implies $X$ is a simple group and of course $X$ is normalized by $b$. If conjugation by $f$ induces an element in the coset of a graph automorphism, then by \cite[Proposition 4.9.2]{GLS3}, $X$ is a simple group or $F^*( C_{N_1}(f))$ is a $p$-group. In the latter case, $a$ normalizes a $p$-group contrary to \cref{prop:not plocal}. So in all three cases $X$ is simple and normalized by $b$, so \cref{lem:sol over} implies that $N$ is a simple group.

    Now, $C_{N}(f)$ is normalized by $a$ and so $T=\gen{a^{C_N(f)}}$ is soluble by \cref{lem:subgrps ok}. In particular, $T$ commutes with the simple group $X$ and \(p\) divides \(\abs{X}\). But then there exists a non-trivial $p$-subgroup which commutes with $a$, contrary to \cref{prop:not plocal}. This completes the proof.
\end{proof}

\begin{Lem} \label{NotClassical}
    Suppose \(N_1\) is a classical group and $a\in K^\#$. Then $N_1 \cong \POmega_8^+(p^c)$ and $N_1\gen{a^n}$ involves a triality automorphism.
\end{Lem}

\begin{proof}
    Set $b= a^n$ and consider $N_1\gen{b}$. If $N_1 \cong \POmega_8^+(p^c)$, then assume that $N_1\gen{b}$ does not involve a triality automorphism. Let $P$ be the image in $N_1$ of a parabolic subgroup of the classical group which fixes a $1$-space fixed by a Sylow $p$-subgroup when acting on the natural module associated with $N_1$. Then, as $N_1\gen{b}$ does not involve a graph automorphism, $P^b \le N_1$ is conjugate to $P$ in $N_1$. If $b$ normalizes $P$, then $b$ normalizes $O_p(P)$ and $a$ normalizes the $p$-group $\gen{O_p(P)^{\gen{a}}}$ which contradicts \cref{prop:not plocal}. Hence $P^b \ne P$. Then, by \cref{Cornice}, there exists an involution $t \in N_1$ such that $P^t= P^b$. But then $bt$ normalizes $P$. Therefore $bt$ normalizes $O_p(P)$ and \cref{prod inv normalizes} implies $\abs{O_p(P)^{\gen{at}}}=n$, so $O_p(P)$ and $O_p(P)^x$ commute whenever $x \in \gen{at} \setminus N_G(O_p(P))$. Hence $at$ normalizes the \(p\)-group $\prod_{x\in \gen{at} }O_p(P)^x$. But then so does $(at)^2=aa^t$. Notice that $at$ does not have order $2$ by \cref{lem:invs3,lem:GneN}. Since there exists $a_*\in K^\#$ such that $\gen{aa^t}= \gen{a_*}$, this contradicts \cref{prop:not plocal}.
\end{proof}

\begin{Lem} \label{lem:here's E6}
    Suppose that $\gen{a} \cap N \ne 1$ for some $a \in K$. Then $N_1 \cong\E_6(p^c)$ or $N_1\cong {}^2\!\E_6(p^c)$.
\end{Lem}

\begin{proof}
    Assume that $N_1 $ is neither $\E_6(p^c)$ nor ${}^2\!\E_6(p^c)$. By \cref{NotClassical}, we know that either $N_1 \cong \POmega_8^+(p^c)$ or $N_1$ is an exceptional group. Let $s \in (\gen{a} \cap N)^\#$. Then $\pi_1(s)$ is a semisimple element of $N_1$ by \cref{prop:not plocal}. By \cref{lem:strgrl}, $\pi_1(s)$ is strongly real. Assume that the involution $t_1\in N_1$ inverts $\pi_1(s)$. Then the involution $t=\prod_{i=0}^{n-1} t_1^{a^i}$ inverts $s$ and so $s$ is strongly real.  \cref{lem:invs3} combined with \cref{lem:GneN} then yields that some element of $K^\#$ centralizes an involution, contradicting \cref{CNa Odd}.
\end{proof}

 When  $N_1$ is one of $\E_6(p^c)$ or ${}^2\!\E_6(p^c)$ we define \(N^*\) such that $N\le N^* \le N\gen{b}$ and $N^*/C_N(N_1)$ identified as a group of automorphisms of $N_1$ which contains all the inner-diagonal automorphisms of $N_1$ which are induced by conjugation by some element of $N\gen{b}$. This means that $N^*/C_N(N_1)$ is a subgroup of $\E_6(p^c).(p^c-1,3)$ or ${}^2\!\E_6(p^c).(p^c+1,3)$ containing the socle.

\begin{Lem} \label{acapN1}
    Suppose that $N_1\cong \E_6(p^c)$ or ${}^2\!\E_6(p^c)$. Then $\gen{a} \cap N^*=1$ for all $a\in K^\#$.  In particular, $\gen{a} \cap N=1$.
\end{Lem}

\begin{proof}
    Suppose that $\gen{a} \cap N^* \ne 1$.  Set $b= a^n$ and $\gen{d}= \gen{a} \cap N^*$. Let $e \in \gen{d}\le \gen{a}$ have prime order $r$. By \cref{prop:not plocal}, $r \ne p$. Set $e_1 = \pi_1(e)$. Then $e_1$ is a semisimple element of order $r$ and $C_{N_1}(e_1)$ is normalized by $b$.

    Using \cite[Theorem 4.2.2]{GLS3}, we know that $C_{N_1}(e_1) $ contains the subgroup $L_1T_1$ of index dividing $3$ where $L_1=O^{p'}(C_{N_1}(e_1))$ is a product of Lie components and $T $ is an abelian $p'$-group. Furthermore, if $\abs{C_{N_1}(e_1):L_1T_1}=3$, then $r=3$ and $L_1$ is non-soluble by \cite[Table 4.7.3]{GLS3}. We first show the following.

    \medskip

    \begin{claim}\label{clmE61}
        The group $C_{N_1}(e_1)$ is soluble. In particular, $r> 3$.
    \end{claim}

    \medskip
    If $n> 1$, then $C_{N_1}(e_1)$ is soluble by \cref{lem:sol over}.

    If $n=1$, then $a$ normalizes $C_N(e_1)$ and $\gen{a^{C_N(e_1)}}$ is a soluble group. If $L_1$ is not soluble, $\gen{a^{C_N(e_1)}}$, thus $a$, commutes with $L_1$. But $L_1$ has even order and this contradicts \cref{CNa Odd}. As such, $L_1$ and hence $C_{N_1}(e_1)$ is soluble. Finally, if $r=3$, then \cite[Table 4.7.3]{GLS3} shows that $C_{N_1}(e)$ is not soluble. This proves \ref{clmE61}.

    \medskip

    Our next objective is to prove the following claim.

    \begin{claim}   \label{clmE62}
        We have $L_1$ is non-trivial and $p^c \le 3$. In particular, $N=N^*$.
    \end{claim}

    \medskip

    Suppose that $L_1=1$. Then $C_{N_1}(e_1)= T_1$ and $T_1$ is a maximal torus. Furthermore, $T_1$ is non-degenerate as $T_1$ contains $e_1$ (see \cite[Proposition 3.6.1]{Carter}). Using \cite[Proposition 3.3.6 and 3.6.4]{Carter}, we obtain $N_G(T_1)/T_1 \cong C_{W,\phi}(w)$ (we shall not need the definition of the latter group). The group structures $N_G(T_1)/T_1$ and $T_1$ are conveniently tabulated in \cite[Tables 6,14]{RowleyETAL}.

    Suppose that $b$ has order divisible by $3$. Then, as all choices of $r$ are greater than $3$, $\gen{b}$ contains an element which acts as a field automorphism of order $3$ on $N_1$. Since this contradicts \cref{no field autos rk2}, we see that all the prime divisors of $\abs{b}$ are greater than $3$.

    If $N_{N_1}(T_1)$ is not soluble, then $n=1$ by \cref{lem:sol over}. If $n=1$, then $\gen{a^{N_{G}(T_1)}}$ is a soluble normal subgroup of $N_{G}(T_1)$ and so $a=b$ centralizes $N_{N_1}(T_1)/T_1$ by \cite[Tables 6, 14]{RowleyETAL}.

    Suppose that $N_{N_1}(T_1)$ is soluble. Then $b$ acts on $N_G(T_1)$ by conjugation, and the order of $b$ has no prime divisor smaller than $5$. From the structures presented in \cite[Tables 6, 14]{RowleyETAL}, the soluble groups $C_{W,\phi}(w)$ do not admit automorphisms of prime order greater than $3$. Hence $b$ and $N_{N_1}(T_1)/T_1$ commute in all circumstances.

    Suppose that $S/T_1 \in \Syl_2(N_{N_1}(T_1)/T_1)$ is non-trivial. Then $b$ normalizes $S$ and acts on the $2$-group $S/O_{2'}(T_1)$. Since $b$ centralizes $S/T_1$, $C_{S/O_{2'}(T_1)}(b)\ne 1$ by coprime action. Let $t\in S$ be an involution such that $t O_{2'}(T_1) \in C_{S/O_{2'}(T_1)}(b) $. Then $bt$ has even order and so does $at$ by \cref{prod inv order}. But then $atat= aa^t$ is non-trivial and commutes with an involution by \cref{lem:GneN}.  Thus $aa^t \in \mathbf D_K^\#$ and so taking $a_*\in K$ such that $\gen{a_*}= \gen{aa^t}$, we have a  contradiction to \cref{CNa Odd}. Hence $N_{N_1}(T_1)/T_1$ has odd order.

    Again referring to \cite[Tables 6, 14]{RowleyETAL}, we see that $N_1/T_1$ is cyclic of order $9$ and $\abs{T_1}= (q^6+\epsilon q^3+1)/(3,q-\epsilon)$ where $\epsilon =+$ if $N_1$ is untwisted and $\epsilon =-$ if $N_1$ is twisted. We now use \cite[Tables 1, 2, 9 and 10]{Craven} to see that there is a unique $N_1$-conjugacy class of subgroups $M_1$ such that
    $$
        N_{N_1}(T_1)\le M_1 \cong
        \begin{cases}
            \PSL_3 (q^3).3  &   \epsilon =+,    \\
            \PSU_3 (q^3).3  &   \epsilon =-.
        \end{cases}
    $$

    Since $N_1$ has a unique conjugacy class of maximal subgroups isomorphic to $M_1$, $M_1^b= M_1^x$ for some $x\in N_1$. Set $X= N_{N_1}(T_1)$. Then $X \le M_1 \cap M_1^b=M_1\cap M_1^x$ and so $X$ and $X^x$ are contained in $M_1^x$. Furthermore, $X$ and $X^x$ are conjugate in $M_1^x$. Hence there exists $c \in M_1^x$ such that $X^{xc}= X$. But $N_{N_1}(X)$ normalizes $C_X(X')=T_1$, which means that $N_{N_1}(X)=X$. Hence $xc\in X\le M_1$ and then, as $c\in M_1^x$,
    $$
        M_1^b=M_1^x=M_1^{xc}= M_1.
    $$
    Hence $b$ normalizes $M_1$ and \cref{lem:sol over} shows that $n=1$. But then $\gen{a^{M_1}} $ is soluble and centralizes $M_1$ and this means that $a$ centralizes an involution, a contradiction. We have demonstrated that $L_1\ne 1$. In particular, the Lie components of $C_{N_1}(e_1)$ must be soluble by \cref{clmE61}. This in turn yields $p^c\le 3$ and verifies the first part of \ref{clmE62}. Finally note that $N=N^*$ as $3$ does not divide the order of $b$.

    \medskip

    By \ref{clmE62} we need to consider $p^c=2$ and $p^c=3$. If $p^c=2$, then we use the character tables of $\E_6(2)$ and ${}^2\!\E_6(2)$ which are available in {\sc GAP} \cite{GAP} to show that there is an involution $t \in N_1$ with $tb$ of order divisible by $4$. Then, using \cref{prod inv order}, $ta$ has order divisible by $4$. But then $atat=aa^t$ has even order, which is a contradiction.

    Hence $N_1 \cong \E_6(3)$ or ${}^2\!\E_6(3)$. These groups have outer automorphism groups of order $2$ and so $b \in N $. Hence $b_1=\pi_1(b) \in C_{N_1}(e_1)$ which is soluble and $L_1\ne 1$ by \ref{clmE61} and \ref{clmE62}. Since $L_1$ is a product of Lie components defined in characteristic $3$, $L_1$ is a $\{2,3\}$-group. Because $b_1$ has order coprime to $3$ by \cref{clmE61}, we deduce that $b_1$ centralizes $L_1$. But then $b$ centralizes an involution and this contradicts \cref{CNa Odd}. We conclude that $d=1$ as claimed.
\end{proof}

\begin{Prop}    \label{not rank at least 2}
    The subgroup \(N_1\) is not the derived group of a Lie type group of Lie rank at least 2.
\end{Prop}

\begin{proof}
    Let $a\in K^\#$. Then, by \cref{cor:wreathgood}, $b=a^n$ is non-trivial. Since $\gen{a} \cap N =1$ by \cref{lem:here's E6,acapN1}, $b$ induces by conjugation a non-trivial non-inner automorphism on $N_1$. By \cref{no field autos rk2}, no element of $\gen{b}$ of prime order induces by conjugation on $N_1$ an elements in the coset of a field, graph-field or graph automorphism in $\Aut(N_1)/\Inn (N_1)$. Applying \cref{NotClassical} implies that $N_1$ is not a classical group (as we have ruled out the graph automorphism) and so the elements of $\gen{b}$ induce, by conjugation, inner-diagonal automorphisms of $N_1$.

    In particular, $N_1$ is an exceptional group which has outer automorphisms which are inner-diagonal of odd order. This means that $b$ has order $3$ and $N_1$ is either $\E_6(p^c)$ or ${}^2\!\E_6(p^c)$. This contradicts \cref{acapN1}.
\end{proof}

\section{Lie-type groups of rank one}\label{sec:rank1}

In this section, we consider the possibility that $N_1$ is a rank one Lie-type group. Hence $N_1$ is $\PSL_2(p^c)$, $\PSU_3(p^c)$, $\Sz(2^c)$ or ${}^2\!\G_2(3^{c})$.
We start with $\PSL_2(p^c)$, $\Sz(2^c)$ and ${}^2\!\G_2(3^{c})$ and then consider $\PSU_3(p^c)$ at the end of the section.

\begin{Lem} \label{lem:rank no autos}
    Suppose that $N_1$ is one of $\PSL_2(p^c)$, $\Sz(2^c)$ or ${}^2\!\G_2(3^{c})$. If $a\in K^\#$, $\gen{a} \cap N \ne 1$.
\end{Lem}

\begin{proof}
    We write $N_1= \G(p^c)$ to indicate any of the rank $1$ groups being investigated.

    Suppose that $a \in K^\#$ with $\gen{a} \cap N= 1$ and set $b=a^n$. By \cref{dividers}, $b \ne 1$. Let $f\in \gen{b}$ have prime order $r$. Then, as $\gen{b}\cap N=1$, conjugation by $f$ induces a field automorphism on $N_1$. Thus $O^{p'}(C_{N_1}(f))$ is the corresponding subfield subgroup $\G(p^{c/r})$ and this subgroup is normalized by $b$.

    We first show

    \begin{claim}   \label{sl2Intneq1}
        $N=N_1$ is simple.
    \end{claim}

    \medskip

    Assume that $n>1$. Then $C_{N_1}(f)$ is soluble by \cref{lem:sol over}. Hence $N_1 \cong \SL_2(2^r)$, $\Sz(2^r)$, or $\PSL_2(3^r)$ and, in addition, $\gen{b}=\gen{f}$.

    Since $b$ induces a field automorphism on $N_1$, $b$ normalizes a Borel subgroup $B_1$ of $N_1$. It follows that $B=\gen{B_1^{\gen{a}}}$ normalizes a Sylow $p$-subgroup $T$ of $N$.

    Since $\gen{f}= \gen{b}$ has order $r$, $a$ has order a power of $r$ by \cref{dividers}. Assume that $p=2$. Then either $N_1\cong \SL_2(2^r)$ and $C_{N_1}(b)\cong \SL_2(2)\cong \Sym(3)$, or $N_1 \cong \Sz(2^r)$ and $C_{N_1}(b)\cong \Sz(2)\cong \Frob(20)$. In particular, $C_{B_1}(b)$ is cyclic of order either $2$ or $4$ in the respective cases. Since $\gen{a}$ acts transitively on $\{N_1, \dots, N_n\}$, $C_B(a)$ is also cyclic.

    Suppose that $N_1\cong \SL_2(2^r)$. Set $F= B\gen{a}$, $H= T\gen{a}$ and put $\ov F= F/T$. We claim that $\ov F$ is a Frobenius group with complement $\ov H$ and kernel $\ov B$. Let $\ov j \in \ov F\setminus \ov H $. Then $\ov j = \ov h {\;}\ov \ell$ where $\ov h \in \ov H$ and $\ov \ell \in \ov B $ with $\ov \ell \ne 1$. Notice that $\ov B$ is normal in $\ov F$. Hence $\ov H \cap \ov H^{\ov{j}}= \ov H \cap \ov H^{\ov \ell}$. Assume that $\ov H \cap \ov H^{\ov \ell}\ne 1$. Then, as $\ov H$ is cyclic of order a power of $r$, $\ov b \in \ov H \cap \ov H^{\ov \ell}\ne 1$ and $\ov \ell$ normalizes $\gen{\ov b}$. But then $[\ov\ell, \ov b] \in \ov H \cap \ov B=1$. However, $C_{\ov B}(\ov b)=1$ and so this contradicts the choice of $\ov \ell$. We conclude that $\ov H \cap \ov H^{\ov {j}}=1$ for all $\ov j \in \ov F\setminus \ov H$. This proves the claim that $\ov F$ is a Frobenius group. As involutions in $C_T(a)$ project non-trivially onto each factor of $N$, we have $C_B(C_T(H))=T$ and so we may apply \cref{Even Order in K^2} to obtain a contradiction. Thus $N_1\not \cong \SL_2(2^r)$.

    If $N_1 \cong \Sz(2^r)$, recalling that $\Omega_1(T)$ is the subgroup of $T$ generated by all the involutions in $T$, we rerun the above argument but working in the group $B\gen{a}/\Omega_1(T)$ and obtain the same contradiction as in the $\SL_2(2^r)$ case.

    We are left with $N_1 \cong \PSL_2(3^r)$ and we know $C_{N_1}(b) \cong \PSL_2(3)\cong \Alt(4)\cong C_N(a)$. Now $\gen{C_{N_1}(b),a}/\gen{b} \cong \Alt(4)\wr C$ where $C=\gen{a}/\gen{b}$ has order $n$. Application of \cref{lem:wreath orders} yields two conjugates of $a$ which have product of even order. This contradiction demonstrates that $n=1$ and proves \cref{sl2Intneq1}.

    Because of \cref{sl2Intneq1}, we can now examine the simple groups $N$ which can arise. We know $N$ is one of $\PSL_2(p^c)$, $\Sz(2^c)$ or ${}^2\!\G_2(3^c)$ with $c$ odd in the last two cases. Furthermore $a$ generates a subgroup of field automorphisms of $N$. Let $f \in \gen{a}$ have prime order $r$. Then $C_N(f)\cong \G(p^{c/r})$. If $a$ acts non-trivially on $C_N(f)$, then $O^{p'}(C_N(f))$ is simple and $\gen{ a^{C_N(f)}} $ is not soluble, which contradicts \cref{lem:subgrps ok}. Hence $\gen{a}= \gen{f}$ and $a$ has order $r$. We shall show that $C_N(a)$ is soluble. If $r$ divides $\abs{C_N(a)}$, then there exists $z\in C_N(a)$ of order $r$ such that $za$ has order $r$ and $\gen{(za)^{C_G(a)}} \ge C_N(a)$. Hence \cite[Proposition 4.9.1]{GLS3} implies $za$ is also a field automorphism of $N$ and that there exists an automorphism $\theta $ of $N$ such that $za\in K\theta$. Now using \cref{lem:autos}, $(G,K\theta)$ is also a minimal counterexample to \cref{C3}. Hence $\gen{K\theta \cap C_N(a)}\ge C_N(a)$ is soluble, from which we conclude that $C_N(a)$ is soluble. Suppose that $r$ is coprime to $\abs{C_N(a)}$. Then, setting $J= a^G \subseteq K$, we find an element $d\in J^2 \subseteq \mathbf D_K$ such that $d$ has order greater than $r$ by \cref{our thm}. Let $d_*\in K$ with $\gen{d_*} =\gen{d}$. Then $d_*$ has order $rs$ where $s$ is coprime to $r$. Writing $d_* =xy$ with $x$ of order $r$ and $y \in O^{p'}(C_N(x))\cong \G(p)$, we find that $\gen{K \cap C_N(x)} \ge \gen{d_*^{C_N(x)}} \ge \gen{y^{C_N(x)}} \ge O^{p'}(C_N(x))$ which is therefore soluble by \cref{lem:subgrps ok}. Since $\gen{a}\in \Syl_r(G)$, $\gen{x}$ and $\gen{a}$ are conjugate. Hence $O^{p'}(C_N(a))$ is conjugate to $O^{p'}(C_N(x))$ and we conclude that $C_N(a)$ is soluble. It follows that $N \cong \SL_2(2^r)$, $\PSL_2(3^r)$, or $\Sz(2^r)$.

    If $r$ does not divide $\abs{G}$, \cref{coprime autos good} yields a contradiction. It finally follows that $r$ divides $\abs{C_N(a)}$. If $N \cong \SL_2(2^r)$ or $\PSL_2(3^r)$, we conclude that $r=3$, if $N \cong \Sz(2^r)$, then $r=5$. These small cases are easily handled using {\sc Magma} \cite{Magma}. This concludes the proof.
\end{proof}

\begin{Lem} \label{notSL22or2B2}
    The group $N_1 $ is not isomorphic to either $\SL_2(2^c)$ or $\Sz(2^c)$.
\end{Lem}

\begin{proof}
    Assume that $N_1 \cong \SL_2(2^c)$ or $\Sz(2^c)$ with $c$ odd in the latter case.

    Let $a \in K$ and set $\gen{d} = \gen{a} \cap N$. Then $d$ is non-trivial by \cref{lem:rank no autos}. Since $a$ has odd order, $\pi_1(d)$ is semisimple. The semisimple elements of $N_1$ are strongly real and so we use \cref{lem:invs3} to conclude that either $G=N$ or some element $a_*$ of $K$ centralizes an involution $t \in N$. In the latter case, $\pi_1(t)$ is also an involution. Thus $C_{N_1}(\pi_1(t))$ is a $2$-group and so $C_N(t)$ is also a $2$-group.  Since $a_*$ has odd order, we deduce that $\gen{a_*} \cap N=1$, and this is against \cref{lem:rank no autos}. Hence $G=N=N_1$. Now, all the elements of $N$ of odd order are regular semisimple. In particular, by \cref{Gow'sThm} there exist $c$, $d \in K$ such that $cd$ has order $2^c-1$. Thus some element $e$ of $K$ normalizes a Sylow $2$-subgroup $T$ of $G$. Since $e$ is conjugate to its inverse in $G$, we can find $u\in T$ with $f=(e^{-1})^u \in K$. But then $fe=[u,e]\in T^\#$ and this is impossible as $fe\in K^2\subseteq \mathbf D_K$ and $\mathbf D_K$ has no elements of even order by \cref{lem: 2 for now}.
\end{proof}

\begin{Lem} \label{psl2pl1}
    Suppose that $p$ is an odd prime. Then $G \not \cong \PSL_2(p^c)$.
\end{Lem}

\begin{proof}
    Set $q=p^c$. Suppose that $a\in K$ has order coprime to $p$, then $a$ divides either $(q-1)/2$ or $(q+1)/2$ and, in particular, $a$ is a regular semisimple element. Hence by Gow's Theorem \cref{Gow'sThm}, $K^2$ contains all the semisimple elements of $G$. Since $p$ is odd, this contradicts $K$ consisting of odd-order elements. It follows that $K$ consists of elements of order $p$. This falls foul of \cref{our thm}.
\end{proof}

\begin{Lem} \label{PSL2ppodd}
    We have $N_1\not \cong \PSL_2(p^c)$ with $p\ge 3$.
\end{Lem}

\begin{proof}
    Set $q= p^c$ and assume that $N_1 \cong \PSL_2(q)$. Let $a \in K$. Then $$G=N \gen{a} >N$$ by \cref{lem:G/N cyclic,psl2pl1}. By \cref{lem:rank no autos}, $\gen{d} =\gen{a} \cap N \ne 1$. Put $b=a^n$ and note that $\gen{a} \ge \gen{b} \ge \gen{d}$. Every element of $N_1$ has order $p$, order dividing $(q-1)/2$ or order dividing $(q+1)/2$. Suppose that $d$ has order dividing $(q-1)/2$ or $(q+1)/2$ and let $Q= \gen{\pi_1(d)^{\gen{a}}}$. As $b$ centralizes $\pi_1(d)$, we see that $Q $ is isomorphic to a direct product of $n$ cyclic groups isomorphic to $\gen{d}$. Since $\pi_1(d)$ is inverted in $N_1$, there exists an involution $t\in N$ which inverts every element of $Q$ by conjugation. Let $X=Q\gen{t}\gen{a}$ and $P=Q\gen{a}$.  Then $P=O_{2'}(X)$  has index $2$ in $X$ and $Q=[P,t]$ is abelian. Hence $\gen{aa^t} \cap Q=1$ by \cref{lem:complements}. By hypothesis, there exists $a_*\in K$ such that $\gen{a_*}=\gen{aa^t}$. Notice that $\gen{a_*} \le Q\gen{a}$. But then $$\gen{a_*} \cap N \le Q\gen{a} \cap N =Q(\gen{a} \cap N)=Q\gen{d}=Q,$$ so that $\gen{a_*} \cap N \le \gen{a_*} \cap Q= 1$, which of course contradicts \cref{lem:rank no autos}. We conclude that

    \begin{claim}   \label{divp}
        for all $w \in K$, $\gen{w} \cap N$ has order $p$.
    \end{claim}

    \medskip
    Assume that $r$ is a prime divisor of $\abs{\gen{b}N /N }$ with $r \ne p$. Then, as $d$ has order $p$, there exists $f \in \gen{b}$ of order $r$ with $f$ acting as a field automorphism on each $N_i$, $1\le i \le n$. Hence $O^{p'}(C_{N_1}(f)) \cong \PSL_2(p^{c/r})$ and $\pi_1(d)\in C_{N_1}(f)$. Now $\gen{a^{C_N(f)} }\ge \gen{\pi_1(d)^{C_{N_1}(f)}}$ is soluble. Thus $p=3$, $r=c$ and $N_1\cong \PSL_2(3^r)$. In particular, $b$ has order $3r$. By \cref{lem:3r to 2r}, there exists an involution $t$ in $N_1$ such that $tb$ has order $2r$. \cref{prod inv order} yields $ta$ of order $2rn$. It follows that $a^ta$ has order $rn$, which contradicts \cref{divp}. Now, applying \cref{dividers} we have shown the following.

    \begin{claim}   \label{fldautos}
        $G> N$ and $\abs{N \gen{b} {:}N}=p^{ y }$ for some $p^{y}$ dividing $c$. Furthermore, every element of $K$ is a $p$-element of order $p\abs{G:N}$.
    \end{claim}

    \medskip

    We intend to apply \cref{jack-3} to obtain our final contradiction. By \cref{divp}, $d$ has order $p$. Therefore $b$ normalizes $Q_1=C_{N_1}(\pi_1(d))$, which is a Sylow $p$-subgroup of $N_1$. Hence $b$ also normalizes $M_1=N_{N_1}(T_1)$. It follows that $a$ normalizes the direct product $M$ of $n$ copies of $M_1$. We know $ M_1/Q_1 $ is cyclic of order $(p^c-1)/2$. Set $H= M\gen{a}$ and let $S \in \Syl_p(H)$ containing $a$. Then $Q=O_p(H)$ is elementary abelian of order $p^{cn}$, $S=Q\gen{a}$ and $S/Q$ is cyclic of order $\abs{a}/p$. Further, $M/Q$ is an abelian normal $p$-complement to $S/Q$ in $H/Q$.

    Assume that $1\ne x\in C_Q(S)$. Then $x\in C_Q(b)$ and $C_Q(b)$ is a direct product of $n$ copies of $C_{Q_1}(b)$ of order $p^{c/p^w}$ by \cref{fldautos}. Since $ C_Q(S)= C_{C_Q(b)}(a)$ and $a$ permutes the $n$ copies of $C_{Q_1}(b)$ transitively, we deduce that $C_Q(a)$ has order $p^{c/p^w}$ with non-trivial elements projecting non-trivially onto each direct factor of $N$. Since $N_1 \cong \PSL_2(p^c)$, we infer that $C_G(x)= S$. We have shown that $C_H(X) = S$ for all $1\ne X \le C_Q(S)$. Since this is condition \eqref{eqn} of \cref{jack-3}, we may apply \cref{jack-3} to see that $K^2$ contains elements of order not equal to $p\abs{G:N}$ and this contradicts \cref{fldautos}. Hence $N_1 \not \cong \PSL_2(p^c)$.
\end{proof}

\begin{Lem} \label{NotReeGroup}
    We have $N_1 \not \cong {}^2\!\G_2(3^c)$ with $c \ge 3$.
\end{Lem}

\begin{proof}
    Set $q=3^c\ge 27$. Let $a\in K^\#$, $b=a^n$ and $\gen{d} = \gen{a}\cap N$. Then $d$ is non-trivial by \cref{lem:rank no autos}.

  Suppose that $G=N$. If $a\in K$ has order a power of $3$, then, as every unipotent element of $G$ is conjugate to an element of the subgroup $H\cong {}^2\!\G_2(3) \cong \PSL_2(8).3$ and $H<G$ (as $G$ is simple), we have a contradiction to \cref{lem:subgrps ok}.  If $a\in K^\#$ is not regular semisimple and not a $3$-element, then $a$ commutes with a $3$-element and so is contained in a Borel subgroup of $G$.  But then \cite[Theorem (3)]{Ward} implies that $a$ has even order, a contradiction.  Hence $a$ is regular semisimple and by \cref{Gow'sThm} there exists $c \in K$ such that $ac$ is an involution.  This contradicts $ac \in \mathbf D_K$ and so we conclude that $G$ is not a simple group.

  Suppose that $t \in C_N(a)$ is an involution. Then $C_{N_1}(\pi_1(t))\cong 2 \times \PSL_2(q)$ is not soluble and is normalized by $b$. \cref{lem:sol over} implies that $n=1$. But then $\gen{a^{C_{N}(t)}}$ is not soluble, a contradiction. Therefore $\abs{C_N(a)} $ is odd for all $a \in K^\#$.

    Since $G>N$, \cref{lem:invs3} yields the following result.

    \begin{claim}   \label{clm:not sr}
        No element of $K$ normalizes a subgroup generated by a strongly real element.
    \end{claim}

    \medskip

    Pick $y \in \gen{d}$ of prime order. \cref{lem:strgrl} and \ref{clm:not sr} together imply that $y$ has order $3$.

    Suppose that $w \in \gen{b}$ has prime order $r$ with $r \ne 3$. Then $C_{N_1}(w)\cong {}^2\!\G_2(3^{c/r})$ is non-soluble. \cref{lem:sol over} tells us that $N=N_1$ is simple. Hence $C_{N_1}(w)$ is $\gen{a}$-invariant and so $\gen{a} \le \gen{K \cap C_G(w)}$. In particular, $\gen{K \cap C_N(w)} $ is not soluble which is a contradiction. We now know that $b$ is a $3$-element and therefore $a$ is a $3$-element by \cref{dividers}. Let $D\in \Syl_3(G)$ with $a\in D$. Then $b$ normalizes $D_1= D \cap N_1$. Therefore $b$ normalizes $Z(D_1)$ which is a root subgroup. The root elements of $N_1$ are strongly real by \cite[Theorem (3)]{Ward}. In particular, $b$ centralizes a strongly real element and consequently so does $a$. This contradicts \ref{clm:not sr} and proves the lemma.
\end{proof}

Finally, we consider the possibility that $N_1 \cong \PSU_3(p^c)$. We use \cite[Theorem 6.5.3]{GLS3} for the subgroups structure of $N_1$ and \cite[II 10.12 Satz]{Huppert} for the structure of a Borel subgroup of $N_1$.

\begin{Lem}\label{notU3q}
We have $G \not \cong \PSU_3(p^c)$.
\end{Lem}
\begin{proof} Set $q=p^c$. Then $q\ge 4$ as $\PSU_3(2)$ is soluble and $\PSU_3(3)\cong \G_2(2)'$ is considered in \cref{not rank at least 2}.

  Suppose that $B$ is a Borel subgroup of $G$.  Set $T =O_p(B )\in \Syl_p(G)$.

\begin{claim}\label{clm:rndivq+1} Suppose that  $a\in (B\cap K)^\#$ and   let $d \in \gen{a}$ have prime order $r$. Then $r$ does not divide $q+1$.  \end{claim}

Suppose that $r$  divides $q+1$. Then $d$ centralizes $Z(T)$. As $q\ge 4$, we deduce that $O^{p'}(C_{N_1} (d) ) \cong \SL_2(q)$.  Since $a \in C_G(d)$ and $A=\gen{K \cap C_G(d)}$ is soluble by \cref{lem:subgrps ok}, $A$ centralizes $O^{p'}(C_{N_1} (d) ) $ and so $A$ is cyclic. Observe that $C_G(d)$ contains a torus of order $(q+1)^2/(3,q+1)$, which has normalizer $M$ of shape $((q+1)^2/(3,q+1)).\Sym(3)$ in $G$ and this group contains $A$. Since $A$ is not normal in $\gen{C_G(d),M}=G$, we conclude that $K \cap C_G(d) \not \subseteq A$, a contradiction. Hence $r$ does not divide $q+1$ and we deduce that $a$ has order coprime to $q+1$.

\medskip

Suppose that $a\in K$ centralizes a non-trivial $p$-element of $G$. Then \cref{Gow'sThm} yields  $c\in a^G\subseteq K$ such that $ac\in \mathbf D_K$ is contained in a Borel subgroup and has order $(q+1)/(q+1,3)$. This contradicts \ref{clm:rndivq+1}. Therefore, every element of $K$ centralizes a non-trivial $p$-element and so $B \cap K\not=\emptyset$.

Let $a \in (K\cap B)^\#$. As $T$ has exponent $p$ when $p$ is odd, we deduce that $a$  has odd order  dividing $(q-1)p$.
Assume that $d \in \gen{a}^\#$ has prime order $r\ne p$. Then  \ref{clm:rndivq+1} implies $r$ divides $q-1$. Thus $C_T(d)=1$ which means that $d$ and hence $a$  do not centralize non-trivial $p$-elements of $G$, a contradiction.

 We conclude that  every element of $K^\#$ has order $p$.    This contradicts \cref{our thm}. We have shown that $G \not \cong \PSU_3(q)$.
\end{proof}

\begin{Lem} \label{NotU3}
    We have $N_1 \not \cong \PSU_3(p^c)$.
\end{Lem}

\begin{proof}
    Set $q=p^c$. Since $N_1$ is not soluble, $q\ge 3$ and, as $\PSU_3(3)\cong \G_2(2)'$, \cref{not rank at least 2} implies that we may assume that $q>3$. By \cref{notU3q}, we know that $G=\gen {a} N >N$ for all $a \in K^\#$. Set $b=a^n$ and let $B_1$ be a Borel subgroup of $N_1$ and $T_1 =O_p(B_1)\in \Syl_p(N_1)$.

If $b$ does not normalize $B_1$, then by \cref{Cornice} there exists an involution $t \in N_1$ such that $bt$ normalizes $B_1$. Then $at$ normalizes $\gen{T_1^{\gen{at}}}$ by \cref{prod inv normalizes}. Therefore, there exists $a_*\in K$ such that $\gen{a_*}=\gen{atat}$ normalizes $\gen{T_1^{\gen{at}}}$. Then $b_*=a_*^n$ normalizes $B_1$. Hence we may assume that $a\in K$ has the satisfies $b=a^n$ normalizes $B_1$.

Suppose that $d \in (\gen{a} \cap N\gen{b})^\#$ has order dividing $q+1$ and acts by conjugation as an inner-diagonal automorphism on $N_1$. Then $d$ normalizes $B_1$ and centralizes $Z(T_1)$. But then $O^{p'}(C_{N_1}(\pi_1(d))) \cong \SL_2(q)$. Since $q>3$, \cref{lem:sol over} implies that $n=1$. Now, arguing as in \ref{clm:rndivq+1}, we find that  $ \gen{K \cap C_{G}(d)} $ is not soluble contrary to \cref{lem:subgrps ok}.

We record this fact.

   \begin{claim}    \label{clm:div q-1}
        For all $a\in K^\#$ which have the property that $b=a^n$ normalizes a Borel subgroup of $N_1$, the elements of $\gen{a} \cap N\gen{b}$ which induce by conjugation inner-diagonal automorphisms on $N_1$ have order dividing $(q-1)q^3$.
    \end{claim}

    \medskip

    Suppose that $d\in \gen{a} \cap N\gen{b}$ induces a prime order inner-diagonal automorphism on $N_1$ by conjugation. Then the order of $d$ either divides $q-1$ or is $p$ by \cref{clm:div q-1}.

    Suppose that $d$ has order dividing $q-1$. Set $C= C_{B_1}(d)$. Then $C$ is cyclic of order $(q^2-1)/(q+1,3)$. Now, $b$ normalizes $C$ and every subgroup of $C$. Let $e\in C$ have order $(q+1)/(3,q+1)$. Then $\gen{e}$ and consequently $O^{p'}(C_{N_1}(e)) \cong \SL_2(q)$ is normalized by $b$. Since $q>3$, we now know that $n=1$ by \cref{lem:sol over}. Hence $a=b$ normalizes $O^{p'}(C_{N_1}(e))$. But $d \in O^{p'}(C_{N_1}(e))$ and we see that $\gen{a^{O^{p'}(C_{N_1}(e))}}$ is non-soluble, a contradiction to \cref{lem:subgrps ok}.

    Claim \cref{clm:div q-1} now implies that $d\in T_1=O_p(B_1)$ has order $p$. In particular, $p$ is odd. Since $b$ normalizes $O_p(B_1)$, for \(j \in B_1 \setminus T_1\) an involution, we have that $bT_1$ commutes with the involution in $jT_1 \in B_1/T_1 $. It follows that $bj$ has even order  (even if $b \in T_1$). By \cref{prod inv order}, $aj$ has even order. Because $j$ is an involution, $ajaj=aa^j\in K^2\subseteq \mathbf D_K$. Notice that by \cref{lem:inv1,notU3q}, $aj$ is not an involution.
    Since $aa^j$ commutes with an involution in $\langle aj\rangle$ and, there exists $a_*\in \gen{aa^j} \cap K$ such that $a_*\in C_N(s)$ for some involution $s\in N$. Hence $\gen{K \cap C_G(s)}$ is a soluble normal subgroup of $C_G(s)$ by \cref{lem:subgrps ok}. Since $O^{p'}(C_{N_1}(\pi_1(s)))$ is not soluble, we have $n=1$. Now, $a_*$ commutes with $O^{p'}(C_{N }(s))\cong \SL_2(q)$ and we infer that $a_*$ has order dividing $q+1$ and induces inner-diagonal automorphisms on $N$. Since $a_*$ commutes with $C_{N_1}(s)$, it normalizes a Sylow $p$-subgroup of $C_{N_1}( s )$ and so \cref{lem:BT} implies that $a_*$ normalizes a Borel subgroup of $N$. But now we have a contradiction to \cref{clm:div q-1} as $a_*$ does not have order dividing $(q-1)q^3$.

    We have shown that $\gen{a} \cap N\gen{b}$ does not contain any elements which act as inner-diagonal automorphisms on $N_1$ by conjugation. Hence $\gen{b}$ acts as a group of field automorphisms on $N_1$. But then $b$ normalizes but does not centralize a subgroup $X\cong\SL_2(q)$, which is not soluble. In particular, $n=1$ by \cref{lem:sol over}. Hence $a$ normalizes $X$, and consequently the soluble group $\gen{a^X}$ centralizes $X$, and this is impossible as $a$ acts non-trivially on $X$.
 \end{proof}

Combining \cref{notSL22or2B2,PSL2ppodd,NotU3,NotReeGroup} yields

\begin{Prop}    \label{notrnk1}
    The subgroup \(N_1\) is not a simple group of Lie type of Lie rank $1$.
\end{Prop}
\qed

\section{The proofs of Theorem~\ref{C3} and its corollaries}\label{sec:proof}

\begin{proof}[Proof of \cref{C3}]
    Suppose that $(G,K)$ is a counterexample to \cref{C3} with $\abs{G}$ minimal. Then by \cref{lem:G/N cyclic} we have $G= \gen{a}N$ for all $a \in K^\#$, where $N$ is a minimal normal subgroup of $G$ which is non-abelian. We can write $N= N_1 \times \dots \times N_n$ where $N_1\cong N_k$, $1\le k \le n$ is a non-abelian simple group. \cref{lem:Alt Gone,notspor} show that $N_1$ is not an alternating group or a sporadic simple group and together \cref{not rank at least 2,notrnk1} show that $N_1$ is not a simple group of Lie type. We conclude that there are no counterexamples to the theorem.
\end{proof}

We now prove \cref{Cor2,CorCam}.

\begin{proof}[Proof of \cref{Cor2}]
    Assume that every element of $K^2$ has odd order. We shall show that $\gen{K}$ is soluble and then deduce that $G$ is soluble. Let $y \in K$. Then, as $G/L$ is abelian, $y^g \in xL$ has odd order for all $g \in G$. Hence $y^g\in K$ and consequently $K$ is a normal subset of $G$. Let $y,z \in K$. Then $y = x\ell_1$ and $z=x\ell_2$ for some $\ell_1, \ell_2\in L$. We have $yz =x^2\ell_1^x\ell_2\in x^2L$. Hence $K^2\subseteq {x^2}L$. Since by assumption $K^2$ consists of odd-order elements and since the map from $G$ to $G$ given by $a\mapsto a^2$ is a bijection between odd-order elements, its restriction to a map from $xL$ to $x^2L$ is also a bijection. Hence
    $$
        K^2 \subseteq \{y^2\mid y \in K\} \subseteq \mathbf D_K
    $$
    and so \cref{C3} implies $\gen{K}$ is soluble. Set $X=\gen{K}$.

    We now claim that $G$ is soluble. If $L \le X$, then we are done, as then $X$ is soluble and $G/X$ is abelian because $G/L$ is assumed to be abelian. Hence $L> L \cap X$. Set $R= L \cap X$, let $\pi$ be the set of prime divisors of the order of our given element $x$ and let $p$ be an odd prime divisor of $\abs{L:R}$. Then $\Pi=\pi \cup \{p\}$ is a set of odd primes.

    Pick $P \in \Syl_p(L)$. Then $PR> R$ and $X$ normalizes $PR$ as $[X,PR]\le [X,L] \le X\cap L=R \le PR$. As $X$ and $PR$ are soluble, so is $PRX$. Let $H$ be a Hall $\Pi$-subgroup of $PRX$ containing $\gen{x}$. Then $H= \gen{x}(H\cap L)$ and $P=R(H\cap L)>R$. In particular, $xh $ has odd order for all $ h \in H\cap L$. Hence $x (H\cap L) \subseteq K \subseteq X$. Since $x^{-1} \in X$, $H \cap L \subseteq X$. Therefore $H \cap L \le X \cap L = R$, and this contradicts $P=R(H\cap L)>R$. It follows that $\abs{ L:R }$ has no odd prime divisors. Hence $L/R$ is a $2$-group, and consequently $G/R$ and $R$ are both soluble. But then $G$ is soluble. This proves the claim.
\end{proof}

\begin{proof}[Proof of \cref{CorCam}]
    Because of \cite[Theorem B]{Camina2020} and \cite[Theorem A]{Beltran2022} we may suppose that $n=2$. Let $x \in K$. Then $x^2 \in K^2$. The hypothesis implies that $x^2 \in D$ or $x^2\in D^{-1}$. Hence we may assume that $D= (x^2)^G$ and $D^{-1}=(x^{-2})^G$. But then for $y,z\in K$, there exists $w \in K$ such that $yz=w^2$ or $yz= w^{-2}$. As $w$ has odd order we have $\gen{yz}= \gen{w^2}=\gen{w}$ or $\gen{yz}= \gen{w^{-2}}=\gen{w}$ and so $yz\in D_w\subseteq \mathbf{D}_K$. Hence $K^2\subseteq \mathbf D_K$ and \cref{C3} implies that $\gen{K}$ is soluble.
\end{proof}

\printbibliography
\end{document}